\documentclass{amsart}

\usepackage{graphicx}
\usepackage{amssymb,amscd, amsthm, latexsym, mathrsfs, epsfig}
\usepackage[all]{xy}

\usepackage{todonotes}

\usepackage{tikz-cd}

\newcommand\PP{\mathbb P}

\newcommand\C{\mathbb C}

\newcommand\Q{\mathbb Q}
\newcommand\R{\mathbb R}
\newcommand\Z{\mathbb Z}
\newcommand\N{\mathbb N}

\newcommand{\X}{\mathcal{X}}

\newcommand{\LL}{\mathcal{L}}
\newcommand{\M}{\mathcal{M}}
\newcommand{\cN}{\mathcal{N}}

\newcommand{\T}{T}

\newcommand{\cZ}{\mathcal{Z}}
\newcommand{\fD}{\mathfrak{D}}
\newcommand{\fd}{\mathfrak{d}}

\newcommand{\eps}{\epsilon}

\newcommand\Aut{\operatorname{Aut}}

\newcommand\bra{\langle}
\newcommand\ket{\rangle}

\newcommand\res{\operatorname{Res}}

\newcommand{\Hom}{\operatorname{Hom}}

\newcommand{\spec}{\operatorname{Spec}}
\newcommand{\NE}{\operatorname{NE}}
\newcommand{\pic}{\operatorname{Pic}}

\newcommand{\jk}{\operatorname{JK}}

\newcommand{\opZ}{\operatorname{Z}}

\makeatletter \@addtoreset{equation}{section} \makeatother

\newtheorem{thm}{Theorem}[section]
\newtheorem{prop}[thm]{Proposition}
\newtheorem{lem}[thm]{Lemma}
\newtheorem{cor}[thm]{Corollary}

\theoremstyle{definition}
\newtheorem{definition}[thm]{Definition}

\newtheorem{rmk}[thm]{Remark}

\newcommand\narrowdots{\hbox to 1em{.\hss.\hss.}}

\title[Log CY surfaces and Jeffrey-Kirwan residues]{Log Calabi-Yau surfaces and Jeffrey-Kirwan residues}

\author{Riccardo Ontani}
\author{Jacopo Stoppa}

\email{rontani@sissa.it}
\email{jstoppa@sissa.it}
\address{SISSA, via Bonomea 265, 34136 Trieste, Italy} 
\address{Institute for Geometry and Physics (iGAP), Via Beirut 2-4, 34014 Trieste, Italy}
\begin{document}
	
	\begin{abstract} We prove an equality, predicted in the physical literature, between the Jeffrey-Kirwan residues of certain explicit meromorphic forms attached to a quiver without loops or oriented cycles and its Donaldson-Thomas type invariants.
	
	In the special case of complete bipartite quivers we also show independently, using scattering diagrams and theta functions, that the same Jeffrey-Kirwan residues are determined by the the Gross-Hacking-Keel mirror family to a log Calabi-Yau surface. 
	
	\end{abstract}
	
	\maketitle

	\setcounter{tocdepth}{1}
	\tableofcontents

\section{Introduction}

In this work we will study a class of Jeffrey-Kirwan residues associated with hyperplane arrangements \cite{brionvergne, szenesvergne}, i.e., suitable periods of explicit meromorphic forms. The particular periods we consider are suggested by the physical literature (especially \cite{beau, beht, cordova}), where they appear in the expression of important quantities for certain physical theories. 
	
	The relevant data for us are given by a quiver $Q$, without loops or oriented cycles, together with a dimension vector $d$. The gauge group is a product of unitary groups $\prod_{i \in Q_0} \operatorname{U}(d_i)$, modulo the diagonal $\operatorname{U}(1)$. The corresponding meromorphic form is defined on the torus $\mathbb{T}(d) = \left(\prod_{i\in Q_0}(\C^*)^{d_i}\right)/\C^*$, and is given explicitly by 
	\begin{align}\label{MeroForms}
		&\nonumber\opZ_Q(d) = \prod_{i\in Q_0}\prod^{d_i}_{s \neq s' = 1} \frac{u_{i,s'} - u_{i, s}}{u_{i, s} - u_{i,s'}-1} \\&\prod_{\{i\to j\} \in \bar{Q}_1}\prod^{d_i}_{s = 1}\prod^{d_j}_{s'=1}\left(\frac{u_{j,s'} - u_{i,s} + 1 - \frac{1}{2}R_{ij}}{u_{i,s} - u_{j,s'} + \frac{1}{2}R_{ij}}\right)^{\bra j, i\ket}  \bigwedge_{(i, s) \in Q_0 \times \{1,\ldots,d_i\} \setminus (\bar{i}, d_{\bar{i}})} du_{i,s},
	\end{align}
	for any fixed ordering of the pairs $(i, s) \in Q_0 \times \{1,\ldots,d_i\} \setminus (\bar{i}, d_{\bar{i}})$. Here $\bar{i} \in Q_0$ is a fixed reference vertex, and $\bar{Q}$, $\bra -, - \ket$ denote the reduced quiver and skew-symmetrised Euler form of $Q$, respectively. The omission of a variable $u_{\bar{i}, N_{\bar{i}}}$ corresponds to quotienting the gauge group by the diagonal $\operatorname{U}(1)$. The variables $R_{ij}$ denote (a priori) arbitrary parameters, known as $R$-charges. For each fixed real stability vector $\zeta = \{\zeta_i,\,i \in Q_0\}$, the physical computations of \cite{beht, cordova} construct a canonical cycle $C \subset \mathbb{T}(d)$, depending on $\zeta$ and $\opZ_Q(d)$, such that the quantity of physical interest can be computed as the period
	\begin{equation}\label{IntroPeriod} 
		\jk(\opZ_Q(d), \zeta) = \left(\frac{1}{2\pi {\rm i}}\right)^{\dim \mathbb{T}(d)}\int_C \opZ_Q(d).
	\end{equation}
	Mathematically, this can be understood as the conjectural identity 
	\begin{equation}\label{DTJK}
		\bar{\chi}_Q(d, \zeta) = \frac{1}{d!}\jk(\opZ_Q(d), \zeta),
	\end{equation} 
	where the left hand side denotes the generalised Donaldson-Thomas invariant of the quiver $Q$ (see e.g. \cite{js, ks}), and the right hand side is the JK (Jeffrey-Kirwan) residue studied in \cite{brionvergne, szenesvergne} and reviewed for our purposes in Section \ref{FlagsSec}. In particular, $\jk(\opZ_Q(d), \zeta)$ should not depend on $R_{ij}$.
	
	In the following, $Z$ denotes a complex valued linear form on $\Z Q_0$, the central charge, inducing the real stability vector $\zeta$. This can be lifted canonically to an element $\hat{Z}$ of the Lie algebra of $\mathbb{T}(d)$, given by $\hat{Z}_{i, s} = Z_i$ for $i \in Q_0$, $s = 1,\ldots, d_i$. Similarly, we write $\widehat{\zeta}$ for the corresponding lift of the real stability vector. Note that the lift $\widehat{\zeta}$ is used in the construction of the JK residue $\jk(\opZ_Q(d), \zeta)$. 
	
	\begin{rmk}\label{regularityRmk} An important point, which does not seem to be addressed in detail in the physical literature, is that the mathematical theory of Jeffrey-Kirwan residues \cite{brionvergne, szenesvergne} requires that the stability vector $\zeta$ satisfies a suitable \emph{regularity} condition with respect to the dimension vector $d$, i.e., the lift $\widehat{\zeta}$ should not lie in the hyperplane arrangement defined by $d$, namely $\{\cup_{i,s,s'} V(u_{i, s} - u_{i,s'})\big)\cup\big( \cup_{i,j,s,s'} V(u_{i,s} - u_{j,s'})\}$ (see Section \ref{FlagsSec} for the details). In this paper we always assume (and keep track of) this regularity condition. 
	\end{rmk} 
	
\subsection{Quivers} Part of our work is devoted to proving a version of the conjectural identity \eqref{DTJK}.
	
We work with \emph{abelianised JK residues}, introduced from physical considerations in \cite{beau}, in the limit of large $R$-charges, which we denote by $\jk^{\infty}_{ab}$. A physical argument (see \cite{beau}, Section 3.4) predicts that in fact we have
	\begin{equation*}
		\jk(\opZ_Q(d), \zeta) = \jk^{\infty}_{ab}(\opZ_Q(d), \zeta).
	\end{equation*} 
	This limitation to abelianised invariants is due to the fact that at present we are only able to prove a crucial iterated residues expansion, \eqref{JKIntro} below, in this case, although we expect that a similar result holds independently of abelianisation.
	
	Let $ Q$ be a quiver without loops or oriented cycles. The abelianised JK residue $\jk_{ab}(\opZ_{ Q}(d), \zeta_{ Q})$, with respect to a regular stability vector $\zeta_{Q}$, is defined in Section \ref{AbelianSec}, following \cite{beau}. It is given by a rational linear combination of the form $\sum_{Q' \in \mathcal{S}(d)} c_{Q'} \jk(\opZ_{Q'}(\mathbb{I}), \zeta_{Q'})$, for a suitable set of quivers $\mathcal{S}(d)$, where $\mathbb{I}$ denotes the full abelian dimension vector. In Sections \ref{SpanningTreeSec} and \ref{AbelianJKSec}, we will show that, for fixed $R$-charges $\bar{R} = \bar{R}_{ij}$ lying in a dense open cone, the JK residue $\jk(\opZ_{Q'}(\mathbb{I}), \zeta_{Q'})$ can be computed as a sum of contributions, one for each spanning tree $T$ of the reduced quiver $\bar{Q}'$, with fixed root $i_{Q'} \in Q'_0$. In brief, the abelianised JK residue can be computed as a linear combination of iterated residues $\operatorname{IR}_0$ (a notion recalled in Appendix \ref{iterResSec}),
	\begin{equation}\label{JKIntro}
		\frac{1}{d!}\jk_{ab}(\opZ_{ Q}(d), \zeta_{ Q}) = \sum_{Q' \in \mathcal{S}(d)} c_{Q'} \sum_{T\subset \bar{Q}'} \eps(T)\operatorname{IR}_{0} (  \phi_{x_T}(\opZ_{Q', \bar{R}}(\mathbb{I})(u))), 
	\end{equation}
	for certain $\eps(T) \in \{0, 1\}$, where $x_T$ denotes a singular point of the affine hyperplane arrangement defined by $(Q', i_{Q'}, \mathbb{I}, \bar{R})$, corresponding to a spanning tree $T$ under a natural bijection, and the map $\phi_{x_T}$ picks a specific identification of $\opZ_{Q', \bar{R}}(\mathbb{I})(u)$ near $x_T$ with a meromorphic function in a neighbourhood of $0 \in \C^{|Q'_0|-1}$ (see Proposition \ref{BETHAbelian} and Corollary \ref{BETHAbelianTrees}). This identity, valid for general $Q$ as above, seems to be interesting in its own right. Indeed, we use it in order to prove our first result, which relates abelianised JK residues to quiver invariants.
	\begin{thm}\label{QuiverChiTheorem} Let $Q$ be a quiver without loops or oriented cycles. Suppose $\widehat{\zeta}$ is regular with respect to the hyperplane arrangement defined by the dimension vector $d$. (This holds, for example, under the stronger condition that $d$ and $\zeta$ are coprime). Then, we have
		\begin{equation}\label{DTJKabelian}
			\frac{1}{d!}\jk^{\infty}_{ab}(\opZ_Q(d), \zeta) = \bar{\chi}_Q(d,\zeta).
		\end{equation}
	\end{thm}
	\begin{rmk} The regularity assumption is restrictive. However, as we explained in Remark \ref{regularityRmk}, this assumption is not a drawback of our approach, but rather is needed for the JK residue $\jk(\opZ_Q(d), \zeta)$ (and so $\jk^{\infty}_{ab}(\opZ_Q(d), \zeta)$) to be well defined, at least according to the existing mathematical literature. Clearly, an extension of the theory beyond this case is desirable.  
	\end{rmk}	
	
	\subsection{Log CY surfaces} The other main insight we take from the physical literature (in particular \cite{beniniS2, closset}) is the expectation that, quite independently of the relation to quiver invariants \eqref{DTJK}, the generating function  
	\begin{equation}\label{JKPartFunc}
		\sum_{d \in \Z_{\geq0}Q_0} \frac{1}{d!}\jk(\opZ_Q(d), \zeta) \prod_{i\in Q_0}\prod^{d_i}_{s = 1} e^{\hat{Z}_{i, s}} 
	\end{equation} 
could be computed in terms of the complex structure of a suitable family of Calabi-Yau manifolds $\mathcal{X} \to S$, such that $e^{\hat{Z}_{i, s}}$ provide coordinates on $S$, the base of the family; and at the same time that, after a well-defined change of variables, \eqref{JKPartFunc} might also encode certain interesting Gromov-Witten invariants. We emphasise that a priori our expectations are only based on a formal analogy with the more complicated Jeffrey-Kirwan generating functions appearing in \cite{beniniS2, closset}, and we do not prove here a direct link to the physical results of these works. 
	
	However, our second result, Theorem \ref{MainThm}, does confirm these expectations for a particular class of Jeffrey-Kirwan generating functions \eqref{JKPartFunc}. That is, we provide a family $\mathcal{X} \to S$ as above and explain precisely how its complex structure determines \eqref{JKPartFunc} in these cases; and how this also yields an interpretation of \eqref{JKPartFunc} in terms of Gromov-Witten theory. 
	
Our approach is based on the mirror symmetry construction for log Calabi-Yau surfaces proposed by Gross, Hacking and Keel\footnote{Here and in the rest of the paper, we reference the unabdriged version of \cite{ghk}, available as {\tt arXiv:1106.4977v1}.} \cite{ghk}. Thus, we consider pairs $(Y, D)$ where $Y$ is a smooth complex projective rational surface, endowed with an anticanonical cycle of rational curves 
	\begin{equation*}
		D = D_1 + \cdots + D_n, 
	\end{equation*}
	(i.e., Looijenga pairs). The complement $U = Y \setminus D$ is log Calabi-Yau: it is endowed with a holomorphic symplectic form $\Omega$, unique up to scaling, with simple poles at infinity.
	
	Suppose that the intersection matrix $(D_i \cdot D_j)$ is not negative semidefinite (i.e., we are in the \emph{positive case} of \cite{ghk}, Section 0.2). Then, by \cite{ghk}, Lemma 5.9, $U$ is affine. The GHK (Gross-Hacking-Keel) mirror to the log Calabi-Yau surface $(U, \Omega)$ is constructed as a family of affine surfaces 
	\begin{equation*}
		\X \to S = \spec \C[\NE(Y)],
	\end{equation*} 
	over the affine toric variety $\spec \C[\NE(Y)]$. The fibres of $\X \to S$ have log-canonical singularities, and a restriction of the family gives a smoothing of the special fibre 
	\begin{equation*}
		\X_0 \cong \mathbb{V}_n := \mathbb{A}^2_{x_1x_2} \cup \cdots \cup \mathbb{A}^2_{x_n x_1} \subset \mathbb{A}^n,
	\end{equation*}
	corresponding to the large complex structure limit. A crucial point is that the restriction to the structure torus,  
	\begin{equation*}
		T_Y = \pic(Y) \otimes \C^* = \spec \C[A_1(Y)] \subset S,
	\end{equation*}
	is a versal family of deformations of $(U, \Omega)$ as a log Calabi-Yau surface. A proof of the latter statement is announced in \cite{ghk}, Section 0.2, as well as in \cite{ghkModuli}, Section 1; recently, a proof has appeared in \cite{LaiZhou}. Thus, in particular, $(U, \Omega)$ appears as a fibre of its own mirror family. 
	
	In \cite{ghk}, Section 0.5, homological mirror symmetry conjectures are formulated for the induced family $\X \to T_Y$. As we discuss more precisely in Section \ref{GHKSec}, it is known that the complex parameters $T_Y$ correspond both to the periods of the holomorphic symplectic form $\Omega$ on the fibres $\X_s$ and to the choice of a complexified K\"ahler class $B + {\rm i}\omega$ on $U$, restricted from $Y$.
	
	Through a remarkable series of reductions (see \cite{ghk}, Sections 1-3), briefly recalled in Section \ref{GHKSec}, the germ of the family $\X \to S$ around $0 \in S$ is shown to be equivalent to the datum of a \emph{consistent scattering diagram} in $\R^2$,
	\begin{equation*}
		\bar{\fD} = \{(\fd, f_{\fd})\}.
	\end{equation*}
	A key step in the proof shows that $\bar{\fD}$ can be obtained as the (essentially unique) consistent completion of an initial, finite scattering diagram $\bar{\fD}_0$, which can be described explicitly. The process of consistent completion can be understood as eliminating all monodromy. Then, the consistent diagram $\bar{\fD}$ determines uniquely a family $\bar{\X}^o$, endowed with canonical regular functions 
	\begin{equation*}
		\bar{\vartheta}_q(t) = \overline{\operatorname{Lift}}_{t} (q),\,t \in \R^2, 
	\end{equation*}
	(defined using a sum over \emph{broken lines}), and the mirror family $\X \to S$, with its theta functions $\vartheta_q$, can be reconstructed from these. The functions $\bar{\vartheta}_q$ satisfy the fundamental identity\footnote{We call our variable $t$, rather than $Q$ as in \cite{ghk}, in order to avoid confusion with quivers.}
	\begin{equation}\label{MonodromyIntro}
		\bar{\vartheta}_q(t^+) = \theta_{\fd, \bar{\fD}}\left(\bar{\vartheta}_q(t^-)\right),  
	\end{equation}
	along each ray $\fd$, where $\theta_{\fd, \bar{\fD}}$ is the automorphism attached to the weight function $f_{\fd}$.  
	
	We can now describe our second result. It is only valid for the class of \emph{complete bipartite quivers} (see Sections \ref{AbelianSec}, \ref{GenMonSec2} and \cite{rw}, Section 5 for background on these). The relevant log Calabi-Yau surfaces have a toric model mapping to $\PP^2$, with boundary given by a triangle of lines $L_1 + L_2 + L_3$. On the other hand, the result does \emph{not} rely on any a priori correspondence between JK residues and quiver invariants (such as Theorem \ref{QuiverChiTheorem}, or the conjectural identity \eqref{DTJK}), but rather shows directly how the Jeffrey-Kirwan generating function is determined by the GHK family: see our outline of the proof in Section \ref{outline}. In brief, we show that the consistent scattering diagram $\bar{\fD}$ for the mirror family $\X \to S$ automatically knows about JK residues.  
	\begin{thm}\label{MainThm} Let $Q$ denote a quiver as above. Suppose $Q$ is complete bipartite, and let $\zeta$ denote a nontrivial compatible stability vector. Then, it is possible to construct an affine log Calabi-Yau surface $U = Y \setminus D$, depending only on $Q$, with toric model mapping to $(\PP^2, L_1 + L_2 + L_3)$, such that the following hold.
		\begin{enumerate}
			\item[$(i)$] The weight function $f_{\fd}$ in the consistent scattering diagram $\bar{\fD}$ for $U$ can be identified canonically with a sum over dimension vectors for $Q$, 
			\begin{equation*}
				f_{\fd} = \exp\left(\sum_{k>0} \sum_{|d| = k m_{\fd}} k c_d z^d\right),
			\end{equation*}  
			with 
			\begin{equation*}
				z^d = \exp\big(2\pi {\rm i} \int_{\tilde{\beta}(d)} \Omega\big),
			\end{equation*}
			for a class $\tilde{\beta}(d) \in H_2(\X_s, \Z)$, determined by the dimension vector $d$. Here, $|d|$ denotes the element of $\N^2$ induced by $d$ and $m_{\fd}$ is the primitive generator of $\fd$.
			\item[$(ii)$] Assume $\widehat{\zeta}$ is regular with respect to the hyperplane arrangement defined by $d$. Let $D$ denote the dimension of the moduli space of stable representations $\mathcal{M}^{\zeta-st}_d(Q)$. Then, in the canonical expansion above, we have
			\begin{align}
				c_d z^d &= (-1)^D \frac{1}{d!}\jk^{\infty}_{ab}(\opZ_Q(d), \zeta) \exp\big(2\pi {\rm i} \int_{\tilde{\beta}(d)} \Omega\big)\label{MainThmJK}\\
				&= N_{\beta(d)}(Y, D) \exp\big(2\pi {\rm i} \int_{ \beta(d)} [B + {\rm i}\omega]\big).\label{MainThmGW}
			\end{align}
			These identities hold \emph{independently} of the equalities relating JK residues to quiver invariants \eqref{DTJK}, \eqref{DTJKabelian}. Here, 
			\begin{align*}
				N_{\beta(d)}(Y, D)= \int_{[\overline{\mathfrak{M}}((\tilde{Y})^{o}/C^o,\beta(d))]^{\rm vir}}1
			\end{align*}
			denotes a relative genus $0$ Gromov-Witten invariant computed on a blowup $\pi\!:\tilde{Y}  \to Y $, with exceptional divisor $C$, with respect to a degree $\beta(d) \in H_2(\tilde{Y}, \Z)$, where $\pi$ and $\beta(d)$ are determined by $d$.
		\end{enumerate}  
	\end{thm} 
	Note that, according to \cite{ghk}, the invariant $N_{\beta}(Y, D)$ should be thought of as enumerating suitable holomorphic discs in the log Calabi-Yau $U$ (see \cite{brini} for recent related results). 
	\begin{rmk}
		The main reason for the limitation to bipartite quivers is our use of some of the results in \cite{rsw, rw}, which are only proved in this case. It should be possible, though highly nontrivial, to extend the methods of this paper to a larger class of quivers and log Calabi-Yau surfaces, by relying on the recent results of \cite{brini}. We also expect that similar results hold in a refined setting, for the quantum mirrors of log Calabi-Yau surfaces studied in \cite{bousseauThesis, bousseauQuantum}.
	\end{rmk}
	From our perspective, the crucial point is proving the first identity \eqref{MainThmJK} in Theorem \ref{MainThm} $(ii)$. From this, the equality with Gromov-Witten invariants \eqref{MainThmGW}, as well as the following result, follow quite easily by known correspondences.  
	\begin{cor}[Alternative proof of Theorem \ref{QuiverChiTheorem} for bipartite quivers]\label{MainCor} Suppose $Q$ is complete bipartite, and $\widehat{\zeta}$ is regular with respect to the hyperplane arrangement defined by $d$. Then, we have
		\begin{equation*}
			\bar{\chi}_Q(d, \zeta) = (-1)^Dc_d = \frac{1}{d!}\jk^{\infty}_{ab}(\opZ_Q(d), \zeta).
		\end{equation*}
	\end{cor}
\begin{rmk}\label{ScatteringRmk} Let us summarise the various logical implications: Theorem \ref{MainThm} shows that the JK residues $\jk^{\infty}_{ab}(\opZ_Q(d), \zeta)$ appear naturally as certain coefficients of the consistent scattering diagram $\bar{\fD}$, at least for complete bipartite quivers, without assuming a priori the identities between JK residues and quiver invariants \eqref{DTJK}, \eqref{DTJKabelian}, but relying instead on the intrinsic features of the scattering diagram (i.e. its theta functions $\bar{\vartheta}_q$, see Section \ref{outline}). The quiver invariants $\bar{\chi}_Q(d,\zeta)$ can then be recovered through known correspondences between such invariants and certain complete scattering diagrams, as in Corollary \ref{MainCor}.

Conversely, one can take as starting point a result characterising some class of complete scattering diagrams in terms of the invariants of acyclic quivers $\bar{\chi}_Q(d,\zeta)$, such as \cite[Theorem 1.5]{BridgelandScattering}. Then, according to Theorem \ref{QuiverChiTheorem}, certain coefficients of these complete scattering diagrams may also be expressed as residues $\jk^{\infty}_{ab}(\opZ_Q(d), \zeta)$ (a priori, the coefficients corresponding to $d$ for which $\widehat{\zeta}$ is regular, see Remark \ref{regularityRmk}); and it seems interesting to ask if this also reflects some formula for the theta functions of such diagrams.        	
\end{rmk}	
	Theorem \ref{MainThm} $(ii)$ makes precise the expectation that the JK generating function \eqref{JKPartFunc}
	can be computed in terms of the family of complex structures on a Calabi-Yau manifold: up to the identification
	\begin{equation*}
		\prod_{i\in Q_0}\prod^{d_i}_{s = 1} e^{\hat{Z}_{i, s}} = \exp\big(2\pi {\rm i} \int_{\tilde{\beta}(d)} \Omega\big),
	\end{equation*}
	(which is possible, by versality), it is determined by the scattering diagram $\bar{\fD}$, and so by the datum of the GHK mirror family around the large complex structure limit $\X_0$. At the same time it is also computed by the (open) genus zero Gromov-Witten theory of $(U, \Omega)$. The relevant change of variable is given by the mirror map since, according to \cite{ghk} (and, more generally, by a result of Ruddat and Siebert, see \cite{ruddat}), the mirror map for the GHK family around $\X_0$ is given by
	\begin{equation*}
		s = \exp(2\pi {\rm i}[B+{\rm i}\omega]).   
	\end{equation*}
	\subsection{Outline of the proofs}\label{outline} The basic ingredient for Theorem \ref{QuiverChiTheorem} is the expansion for JK invariants \eqref{JKIntro}, which we establish in Sections \ref{SpanningTreeSec} and \ref{AbelianJKSec}, after providing some necessary background in Sections \ref{FlagsSec} and \ref{JKquivSec}. The proof of Theorem \ref{QuiverChiTheorem} is completed in Section \ref{QuiverChiProofSec}.
	
	We outline the proof of the first identity \eqref{MainThmJK} in Theorem \ref{MainThm} $(ii)$. On the log Calabi-Yau side, the main difficulty is understanding why an expansion such as \eqref{JKIntro}, involving iterated residues, would appear in the coefficients of the consistent scattering diagram $\bar{\fD}$ in $\R^2$, independently of \eqref{DTJK} and \eqref{DTJKabelian}.  
	
	The key observation is that the regular functions $\bar{\vartheta}_q$ satisfy the formal property \eqref{MonodromyIntro} which characterises the flat sections of a meromorphic connection $\nabla_{\bar{\fD}}$, defined on $\PP^1 = \C^* \cup \{0\} \cup \{\infty\}$, with singularities at $0$ and $\infty$, and with a suitable structure group containing the automorphisms $\theta_{\fd, \bar{\fD}}$. The automorphisms $\theta_{\fd, \bar{\fD}}$ then appear as the generalised monodromy of $\nabla_{\bar{\fD}}$ at $0$. Here we use the standard identification of $\C$ with the real vector space $\R^2$ containing $\bar{\fD}$. This observation is by no means new: in fact, it appears already in the work of Gross, Hacking and Keel, see in particular the discussion on p. 27 and in Section 5.3 of \cite{ghk}. Indeed, flat connections of this type have been constructed and studied in the literature, for various structure groups, see in particular \cite{bt_stab, fgs, gmn}. The reference \cite{fgs}, which we will follow, is closest to the setup of \cite{ghk}. 
	
	Let $Q$ be a complete bipartite quiver, with nontrivial compatible stability vector $\zeta_{Q}$. In Sections \ref{GenMonSec} and \ref{GenMonSec2}, we define a log Calabi-Yau surface $U$, with toric model mapping to $(\PP^2, L_1 + L_2 + L_3)$, whose deformation type depends only on $Q$. Let $\bar{\fD}_0$, $\bar{\fD}$ denote the corresponding initial scattering diagram in $\R^2$ and its consistent completion. We construct a meromorphic connection $\nabla_{\bar{\fD}_0}$ on $\PP^1$, with monodromy corresponding to $\bar{\fD}_0$, together with a distinguished basis of flat sections $\hat{\vartheta}_q$. Moreover, we show that $\nabla_{\bar{\fD}_0}$, $\nabla_{\bar{\fD}}$ fit into a holomorphic family $\nabla$ of meromorphic connections, with constant generalised monodromy. Therefore, the flat sections $\hat{\vartheta}_q$ can be analytically continued to a basis of flat sections for $\nabla_{\bar{\fD}}$. Note that $\hat{\vartheta}_q$, $\bar{\vartheta}_q$ are different. But, by construction, the sections $\hat{\vartheta}_q$ still satisfy
	\begin{equation*} 
		\hat{\vartheta}_q(t^+) = \theta_{\fd, \bar{\fD}}\big(\hat{\vartheta}_q(t^-)\big),  
	\end{equation*}
	so the weight functions $f_{\fd}$ can be computed using the $\hat{\vartheta}_q$. Thus, the process of consistent completion from $\bar{\fD}_0$ to $\bar{\fD}$ can be understood as a process of analytic continuation from $\nabla_{\bar{\fD}_0}$ to $\nabla_{\bar{\fD}}$.
	
	Making this process explicit leads to the expression for $c_d z^{d}$ in terms of iterated residues, appearing in Theorem \ref{MainThm} $(ii)$. More precisely, recall the expansion \eqref{JKIntro} for JK residues is expressed as a sum over quivers $Q' \in \mathcal{S}(d)$. For each spanning tree $T \subset \bar{Q}'$, we introduce a meromorphic function $\operatorname{W}_T$ given by 
	\begin{equation*}
		\operatorname{W}_T(w) = \prod_{\{i \to j\} \in T_1}   \frac{w_i}{w_j} \frac{\bra i, j\ket}{w_j - w_i}.
	\end{equation*}
	Then, we show that $\hat{\vartheta}_q$, as a section of $\nabla_{\bar{\fD}_0}$, has an explicit expression in terms of \emph{iterated integrals} of $\operatorname{W}_T$, of the form
	\begin{equation*}
		\sum_d  \bra q, d\ket z^d \sum_{Q' \in \mathcal{S}(d)} c_{Q'} \sum_{T\subset \bar{Q}'}  \eps(T) \int_{C_T} \operatorname{W}_T(w)
	\end{equation*}
	(see \eqref{explicitY} and \eqref{ExplicitYabelian}). Such formulae in terms of iterated integrals are rather standard for inverse monodromy problems. As a consequence, we find an expression for the analytic continuation of $\hat{\vartheta}_q$ in terms of \emph{iterated residues} of $\operatorname{W}_T$, which leads to the identity
	\begin{align*} 
		c_d z^d &= \sum_{Q' \in \mathcal{S}(d)} c_{Q'} \sum_{T\subset \bar{Q}'}  \eps(T)\operatorname{IR}_{0}(\phi(\operatorname{W}_T(w))) \exp\big(2\pi {\rm i} \int_{\tilde{\beta}(d)} \Omega\big). 
	\end{align*} 
	Here, $\operatorname{W}_T(w)$ is singular along the hyperplane arrangement defined by $(T, i_{Q'}, \mathbb{I})$, and $\phi$ fixes an identification of $\operatorname{W}_T(w)$ with a meromorphic function in a neighbourhood of $0 \in \C^{|Q'_0|-1}$. Finally, in Section \ref{CompletionSec}, we take the limit of large $R$-charges, and show
	\begin{equation}\label{LargeRIntro}
		\lim_{\lambda \to +\infty} \operatorname{IR}_{0} (  \phi_{x_T}(\opZ_{Q', \lambda\bar{R}}(\mathbb{I})(u)))  = \operatorname{IR}_{0}(\phi(\operatorname{W}_T(w ))) .
	\end{equation}
	By \eqref{JKIntro} this will complete the proof of the first equality \eqref{MainThmJK} in Theorem \ref{MainThm} $(ii)$.
	\begin{rmk} The references \cite{chan, leung} show how to compute the consistent completion $\bar{\fD}$ using a solution $\Phi_{\bar{\fD}_0}$ of the Maurer-Cartan equation, in a dgLa of smooth forms, constructed from the initial diagram $\bar{\fD}_0$. The functions $\bar{\vartheta}_q$ also appear there as (limits of) flat sections for the deformed differential $d_{\Phi_{\bar{\fD}_0}} = d + [\Phi_{\bar{\fD}_0},\,-\,]$ (see in particular \cite{leung} Theorem 3.15). Perhaps, this could be though of as a ``de Rham model" which computes $\bar{\fD}$ in the smooth category, as opposed to a ``Betti model" using broken lines (which are affine linear and combinatorial objects), and a ``Dolbeault model" using the holomorphic object $\nabla$. It should be possible to prove Theorem \eqref{MainThm} (and, perhaps, more general results) in this ``de Rham model".
	\end{rmk}
	
	\noindent{\textbf{Acknowledgements.}} We thank G. Borot, V. Fantini, S. A. Filippini, N. Sibilla and R. P. Thomas for some discussions related to the present paper. We are also grateful to the anonymous Referees    for several suggestions on how to improve the paper, and in particular we thank one of the Referees for a comment leading to our Remark \ref{ScatteringRmk}. 
	
	\section{JK residues and flags}\label{FlagsSec}
	
	In this Section we introduce Jeffrey-Kirwan residues of hyperplane arrangements through their characterisation in terms of flags, started in \cite{brionvergne} and studied systematically in \cite{szenesvergne}. 
	
	Given an $n$-dimensional real linear space $\mathfrak{a}$ with a fixed full rank lattice $\Gamma$, we consider the duals $V := \mathfrak{a}^\ast$, $\Gamma^\ast \subseteq V$. Fix a finite set of generators $\mathfrak{A}$ for $\Gamma^\ast$, and suppose it is projective (i.e., contained in a strict half-space of $V$). We also fix elements $f_1, \ldots, f_n \in \mathfrak{A}$ giving an ordered basis of $V$, and the top form on $\mathfrak{a}$ defined as $d\mu := f_1 \wedge \cdots \wedge f_n$. Given a subset $S \subseteq V$, we will denote by $\mathfrak{B}(S)$ the set of all distinct bases of $V$ consisting of elements of $S$.
	\begin{definition}
		Let $x \in \mathfrak{a}_\mathbb{C}$. We say that $x$ is regular (with respect to $\mathfrak{A}$) if it is not contained in the union of hyperplanes
		\begin{align*}
			\bigcup_{f \in \mathfrak{A}\otimes_\mathbb{R}\mathbb{C}}  V(f) \subset \mathfrak{a}_\mathbb{C}.
		\end{align*}
		The set of regular points of $\mathfrak{a}_\mathbb{C}$ is a dense open cone, denoted by $\mathfrak{a}_\mathbb{C}^{reg}$.
	\end{definition}
	We also need to introduce the dual notion. 
	\begin{definition}
		Let $S \subseteq V$ be a finite set and let $\zeta \in V$. We say that $\zeta$ is $S$-regular if
		\begin{align*}
			\zeta \notin \bigcup_{\substack{I \subseteq S\\ | I | = n-1}} \text{Span}_\mathbb{R}(I).
		\end{align*}
		In particular, we will say that $\zeta$ is \textit{regular} (and write $\zeta \in V^{reg}$) if it is $\mathfrak{A}$-regular. The connected components of the dense open cone $V^{reg}$ are called ($\mathfrak{A}$-)\textit{chambers of $V$}. Moreover, denoting by $\Sigma\mathfrak{A}$ the set of all sums of distinct elements of $\mathfrak{A}$, we say that $\zeta$ is \textit{sum-regular} if it is $\Sigma \mathfrak{A}$-regular. 
	\end{definition}
	The classical Jeffrey-Kirwan residue associated with a hyperplane arrangement (see \cite{brionvergne}), which we denote here with $\text{J-K}_{\xi}^\mathfrak{A}$, is a $\mathbb{C}$-linear functional, depending on a chamber $\xi$, defined on the linear space of germs of meromorphic functions $f$ which are regular in a neighbourhood of the origin in $\mathfrak{a}_\mathbb{C}^{reg}$. 
	\begin{rmk} Following \cite{szenesvergne} we fix $d\mu$ and work with functions $f$ rather than forms $f d\mu$. It turns out that the relevant residue operations do not depend on the fixed choice of $d\mu$, see Remark \ref{orientationRmk}. To make contact with the Introduction, one should choose the constant holomorphic volume form corresponding to the real volume form $d\mu$.
	\end{rmk}
	The original definition of $\text{J-K}_{\xi}^\mathfrak{A}$, which we do not recall here, involves a formal Laplace transform. According to \cite{szenesvergne}, Lemma 2.2, there is a cycle $C$ in $\mathfrak{a}_\mathbb{C}^{reg}$, depending on $\xi$, such that 
	\begin{equation*}
		\text{J-K}_{\xi}^\mathfrak{A}(f) = \left(\frac{1}{2\pi {\rm i}}\right)^{n}\int_C f d\mu
	\end{equation*} 
	for all meromorphic functions $f$ which are regular in a neighbourhood of the origin in $\mathfrak{a}_\mathbb{C}^{reg}$ (cf. \eqref{IntroPeriod}).  
	\begin{definition} The (finite) set $\mathcal{FL}(\mathfrak{A})$ is given by flags of $V$ of the form
		\begin{align*}
			\mathfrak{F} = [\lbrace 0 \rbrace = F_0 \subset F_1 \subset \ldots \subset F_n = V] 
		\end{align*}
		such that, for all $j=0, \ldots, n$, $\mathfrak{A}$ contains a basis of $F_j$ and $\text{dim}_\mathbb{R}F_j = j$. 	\end{definition}	
	
	\begin{lem}[\cite{szenesvergne} p. 11]\label{basisLem} Let $\mathfrak{F} \in \mathcal{FL}(\mathfrak{A})$. Then, there exists a basis $\gamma^\mathfrak{F}$ of $V$ such that the following conditions hold.
		\begin{enumerate}
			\item[$(i)$] $\gamma^\mathfrak{F} \subseteq \mathbb{Q} \cdot \Gamma^\ast$.
			\item[$(ii)$] $\left\lbrace	\gamma^\mathfrak{F}_j \right\rbrace_{j=1}^m$ is a basis of $F_m$, for $m=0, \ldots, n$.
			\item[$(iii)$] $\gamma^\mathfrak{F}_1 \wedge \ldots \wedge \gamma^\mathfrak{F}_n = d\mu$.
		\end{enumerate}
	\end{lem}
	Consider a flag $\mathfrak{F} \in \mathcal{FL}(\mathfrak{A})$. Pick a basis $\gamma^\mathfrak{F}$ as in the Lemma above. This basis induces an isomorphism of $\mathbb{C}$-linear spaces $\mathfrak{a}_\mathbb{C} \cong \mathbb{C}^n$ and so an isomorphism between the germs of meromorphic functions at the origin of $\mathfrak{a}_\mathbb{C}$ and of $\mathbb{C}^n$, which we denote by $\mathcal{M}_{\mathfrak{a}_\mathbb{C},0}$, respectively $\mathcal{M}_{\mathbb{C}^n,0} $.
	\begin{definition} The \textit{flag residue morphism} associated with the flag $\mathfrak{F}$ is the composition
		\begin{align*}
			\mathcal{M}_{\mathfrak{a}_\mathbb{C},0} \xrightarrow{\sim} \mathcal{M}_{\mathbb{C}^n,0}   \xrightarrow{\text{IR}_0} \mathbb{C}.
		\end{align*}
		Here $\text{IR}_0$ denotes the iterated residue operation, recalled in Appendix \ref{iterResSec}. By \cite{szenesvergne}, Lemma 2.5, this does not depend on the choice of the basis $\gamma^\mathfrak{F}$, so we have a well defined map $\operatorname{Res}_{(\cdot)} : \mathcal{FL}(\mathfrak{A}) \rightarrow \left( \mathcal{M}_{\mathfrak{a}_\mathbb{C},0} \right)^\ast$.
	\end{definition}
	A flag $\mathfrak{F} \in \mathcal{FL}(\mathfrak{A})$ determines a partition of $\mathfrak{A}$ into the subsets 
	\begin{equation*}\left \lbrace \mathfrak{P}_j := \mathfrak{A} \cap F_{j} \setminus F_{j-1} \right \rbrace_{j=1}^n.
	\end{equation*} 
	We define the vectors
	\begin{align*}
		\kappa^\mathfrak{F}_j := \sum_{s=1}^j \sum_{\alpha \in \mathfrak{P}_s} \alpha \in \Sigma \mathfrak{A},\,  j=1, \ldots, n.
	\end{align*}
	Note that $\kappa := \kappa^\mathfrak{F}_n$ does not depend on the flag $\mathfrak{F}$.
	\begin{definition} A flag $\mathfrak{F} \in \mathcal{FL}(\mathfrak{A})$ is called \textit{proper} if the vectors $\lbrace \kappa^\mathfrak{F}_j\rbrace_{j=1}^n$ form a basis of $V$.
		We introduce the function $\nu : \mathcal{FL}(\mathfrak{A}) \longrightarrow \lbrace 0, \pm 1 \rbrace$ given by
		\begin{align*}
			\nu(\mathfrak{F}) :=  \begin{cases} 
				0 & \text{if } \mathfrak{F} \text{ is not proper}, \\
				1 & \text{if } \mathfrak{F} \text{ is proper and } \lbrace \kappa_1^\mathfrak{F},  \ldots, \kappa_r^\mathfrak{F} \rbrace \text{ is positively oriented w.r.t. } d\mu,\\
				-1\,\, & \text{if } \mathfrak{F} \text{ is proper and } \lbrace \kappa_1^\mathfrak{F},  \ldots, \kappa_r^\mathfrak{F} \rbrace \text{ is negatively oriented w.r.t. } d\mu.
			\end{cases}
		\end{align*}
		
	\end{definition}
	
	For each flag $\mathfrak{F}$ we introduce the cone
	\begin{align*}
		\mathfrak{c}^+(\mathfrak{F}, \mathfrak{A}) := \sum_{j=1}^n \mathbb{R}_{>0} \kappa_j^\mathfrak{F}.
	\end{align*}
	If $\zeta \in V$ is a sum-regular vector, we define the set of flags for which $\zeta$ is in this cone:
	\begin{align*}
		\mathcal{FL}^+(\mathfrak{A}, \zeta) := \left\lbrace \mathfrak{F} \in \mathcal{FL}(\mathfrak{A})\,:  \zeta \in \mathfrak{c}^+(\mathfrak{F}, \mathfrak{U}) \right \rbrace.
	\end{align*}
	
	\begin{lem} Suppose $\zeta$ is sum-regular. Then all the flags in $\mathcal{FL}^+(\mathfrak{A}, \zeta)$ are proper and hence map to $\pm 1$ through $\nu$.
	\end{lem}
	\begin{proof}
		Let $\mathfrak{F} \in \mathcal{FL}^+(\mathfrak{A}, \zeta)$. Then we have $\zeta = \sum_{j=1}^n c^j k_j^\mathfrak{F}$. If the vectors $\kappa_j^\mathfrak{F}$ are linearly dependent, then they span a space contained in some hyperplane generated by elements of $\mathfrak{A}$. But $\zeta$ is sum-regular by assumption, so this is a contradiction.
	\end{proof}
	
	We can now define a notion of Jeffrey-Kirwan residue for arbitrary meromorphic functions.
	\begin{definition}\label{JKDef}
		Let $\zeta$ be a sum-regular vector of $V$. We define the JK residue map $\jk^\mathfrak{A}_\zeta : \mathcal{M}_{\mathfrak{a}_\mathbb{C},0} \rightarrow \mathbb{C}$ as
		\begin{align*}
			\jk^\mathfrak{A}_\zeta(f) = \sum_{\mathfrak{F}\in \mathcal{FL}^+(\mathfrak{A}, \zeta)} \nu(\mathfrak{F})\operatorname{Res}_\mathfrak{F}(f).
		\end{align*}
	\end{definition}
	
	\begin{rmk}
		Notice that this residue \textit{depends} on the choice of $\zeta$ and, in particular, it is not constant if we let $\zeta$ vary inside a chamber $\xi$ of $V$. This is the price to pay in order to extend the classical J-K residue (as defined in \cite{brionvergne}) to a functional acting on \textit{all} germs of meromorphic functions at the origin of $\mathfrak{a}_\mathbb{C}$.
	\end{rmk}
	
	\begin{rmk}\label{orientationRmk} One checks that for each $\mathfrak{F}\in \mathcal{FL}^+(\mathfrak{A}, \zeta)$, the factors $\nu(\mathfrak{F})$, $\operatorname{Res}_\mathfrak{F}(f)$ depend on $d\mu$ only through the choice of orientation. Moreover, both factors simply change sign when $d\mu $ is replaced with $-d\mu$. This is obvious for $\nu(\mathfrak{F})$ and follows from Lemma \ref{basisLem} and the change of variables formula for iterated residues, Proposition \ref{ResidueAndChangeOfVariables}, in the case of $\operatorname{Res}_\mathfrak{F}(f)$. Thus, the quantity $\jk^\mathfrak{A}_\zeta(f)$ does not depend on the fixed choice of $d\mu$ used in Definition \ref{JKDef}.  
	\end{rmk}	
	The classical Jeffrey-Kirwan residue can now be expressed in terms of the flag residues introduced above.
	
	\begin{thm}[\cite{szenesvergne} Theorem 2.6]\label{JKFlag} Let $\xi$ be a chamber of $V$ and let $\zeta \in \xi$ be a sum-regular vector. Then, for every $f \in \mathcal{M}_{\mathfrak{a}_\mathbb{C},0}$ which is regular (locally around the origin) on $\mathfrak{a}_\mathbb{C}^{reg}$, we have
		\begin{align*}
			\operatorname{J-K}^\mathfrak{A}_\xi(f) = \jk^\mathfrak{A}_\zeta(f)
		\end{align*}
	\end{thm}
	
	When $\mathfrak{A}$ is a basis of $V$, the JK residue takes a special form.
	\begin{lem}
		Let $\mathfrak{A}$ be a basis of $V$ and let $\zeta \in V$ be a $\Sigma \mathfrak{A}$-regular vector. If $\zeta \notin \operatorname{Span}_{\mathbb{R}_{>0}}\{\mathfrak{A}\}$ then $\jk_{\zeta}^{\mathfrak{A}}=0$ as a morphism. Otherwise $\jk_{\zeta}^{\mathfrak{A}}$ is the composition
		\begin{align*}
			\mathcal{M}_{\mathfrak{a}_\mathbb{C},0} \xrightarrow{\sim} \mathcal{M}_{\mathbb{C}^n,0}\xrightarrow{\operatorname{IR}_0} \mathbb{C}.
		\end{align*}
		where the first isomorphism is given by the choice of the unique ordering of $\mathfrak{A}$ such that the components $c_1, \ldots, c_n$ of $\zeta$ with respect to this ordered basis satisfy
		\begin{align*}
			0 < c_n < \cdots < c_1.
		\end{align*}
	\end{lem}
	\begin{proof}
		Choose an ordering $f_1, \ldots, f_n$ for the elements of $\mathfrak{A}$ such that this ordered basis is positively oriented. The set of flags $\mathcal{FL}(\mathfrak{A})$ is in bijection with the set of permutations $S_n$ via
		\begin{align*}
			S_n \ni \sigma \mapsto \mathfrak{F}_\sigma \in \mathcal{FL}(\mathfrak{A}), \qquad (\mathfrak{F}_\sigma)_j := \text{Span}_\mathbb{R}(f_{\sigma(1)}, \ldots, f_{\sigma(j)}).
		\end{align*}
		Every flag $\mathfrak{F}_\sigma$ is proper, since the set
		\begin{align}\label{SigmaBasis}
			\big\lbrace f_{\sigma(1)}, f_{\sigma(1)} + f_{\sigma(2)}, \ldots, \sum_{i=1}^n f_{\sigma(i)} \big\rbrace
		\end{align}
		is a basis of $V$. Note moreover that the orientation of this basis coincides with $(-1)^\sigma$.
		Assume that $\zeta \in \mathfrak{c}^+(\mathfrak{F}_\sigma, \mathfrak{A})$. Then its coordinates $\zeta_1, \ldots, \zeta_n$ with respect to the basis (\ref{SigmaBasis}) are all positive. On the other hand its coordinates $c_1,\ldots, c_n$ with respect to the basis $f_{\sigma(1)}, \ldots, f_{\sigma(n)}$ are given by
		\begin{align*}
			c_{j} := \sum_{i=1}^j \zeta_i.
		\end{align*}
		Thus $\zeta \in \mathfrak{c}^+(\mathfrak{F}_\sigma, \mathfrak{A})$ implies that $0 < c_n < \ldots < c_1$. This condition cannot be satisfied if $\zeta \notin \operatorname{Span}_{\mathbb{R}_{>0}}\{\mathfrak{A}\}$, hence we obtain the first part of the claim. 
		
		On the other hand, if $\zeta \notin \operatorname{Span}_{\mathbb{R}_{>0}}\{\mathfrak{A}\}$, then its components with respect to $f_1, \ldots, f_n$ are all positive and distinct since $\zeta$ is sum-regular. So there exist a unique permutation $\sigma \in S_n$ such that $0 < c_n < \ldots < c_1$, and we have
		\begin{align*}
			\text{JK}_\zeta^\mathfrak{A}(\ast) = \nu(\mathfrak{F}_\sigma)\text{Res}_{\mathfrak{F}_\sigma}(\ast).
		\end{align*}
		We have already seen that $\nu(\mathfrak{F}_\sigma) = (-1)^\sigma$, and that the same number also gives the orientation of the basis $f_{\sigma(1)}, \ldots, f_{\sigma(n)}$. By definition, in order to compute the flag residue, we should express the meromorphic function under the linear change of basis to $(-1)^\sigma f_{\sigma(1)}, f_{\sigma(2)}, \ldots, f_{\sigma(n)}$. By Proposition \ref{ResidueAndChangeOfVariables} we see that this is equivalent to multiplying by $(-1)^\sigma$ the residue obtained by considering the basis $f_{\sigma(1)}, \ldots, f_{\sigma(n)}$. Then the result follows from $(-1)^\sigma (-1)^\sigma = 1$.
	\end{proof}
	\begin{cor}\label{JKBasisCor}
		Let $\mathfrak{A}$ be a basis of $V$ and let $\xi$ be a chamber. If $\xi =\sum_{u \in \mathfrak{A}}\R_{>0}  u$ then the functional $\operatorname{J-K}_{\xi}^{\mathfrak{A}}$ is the composition
		\begin{align*}
			\mathcal{M}_{\mathfrak{a}_\mathbb{C},0} \xrightarrow{\sim} \mathcal{M}_{\mathbb{C}^n,0}\xrightarrow{\operatorname{IR}_0} \mathbb{C}
		\end{align*}
		restricted to the set of meromorphic functions that are holomorphic on $\mathfrak{a}_\mathbb{C}^{reg}$ in a neighborhood of the origin.
		Here the first isomorphism is induced by an arbitrary choice of an ordering for the elements of $\mathfrak{A}$.
	\end{cor}
	\begin{proof}
		By Theorem \ref{JKFlag} we see that, chosen a sum-regular $\zeta \in \xi$, $\text{J-K}_\xi^\mathfrak{A}$ coincides with the restriction of $\text{JK}_\zeta^\mathfrak{A}$. Then the claim follows from the previous Lemma.
	\end{proof}
	
	\section{JK residues and quivers}\label{JKquivSec}
	In this Section we describe the application of JK residues to quiver invariants, following \cite{beau, cordova}. 
	
	Let $ Q$ be a quiver without loops or oriented cycles. We define an equivalence relation on $ Q_1$ by 
	\begin{align*}
		\alpha \sim \beta \iff (t(\alpha), h(\alpha)) = (t(\beta), h(\beta)),
	\end{align*}
	and consider a new quiver $\bar{ Q}$, called the reduced quiver of $ Q$, defined by $\bar{ Q}_0 =  Q_0$ and $\bar{ Q}_1 :=  Q_1/\sim$, with head and tail functions induced on the quotient by those of $ Q$. For every arrow $\alpha \in\bar{ Q}_1$ we define the multiplicity $m_\alpha = | \pi^{-1}(\alpha)| = \langle h(\alpha), t(\alpha)\rangle$, where $\langle - , - \rangle$ is the skew-symmetrised Euler form of $ Q$. 
	
	Fix a dimension vector $d \in \Z_{\geq 0}  Q_0$. The complexified gauge group $G_d$ is a Lie group with Lie algebra $\mathfrak{g}$, and its Cartan subalgebra is the abelian Lie algebra of diagonal matrices
	\begin{align*}
		\mathfrak{h} = \bigoplus_{v \in Q_0} \mathbb{C}^{d_v} \leq \mathfrak{g} = \prod_{v \in  Q_0} \text{Mat}_{d_v \times d_v}(\mathbb{C}).
	\end{align*}
	The roots of $\mathfrak{g}$ are the $\mathbb{C}$-linear functionals $\mathfrak{g} \rightarrow \mathbb{C}$ of the form
	\begin{align*}
		r_{v,i,j} (M) := (M_v)_{j,j} - (M_v)_{i,i}, \text{ for } v \in  Q_0, \, i\neq j \in \lbrace 1, \ldots, d_v \rbrace.
	\end{align*}
	The weights of the representation of Lie groups 
	\begin{align*}
		G_d \longrightarrow \text{GL}(\text{Rep}_d( Q))
	\end{align*}
	given by basechange are the $\mathbb{C}$-linear functionals $\mathfrak{h} \rightarrow \mathbb{C}$ of the form
	\begin{align*}
		\rho_{\alpha, i,j}(M) := (M_{h(\alpha)})_{j,j} - (M_{t(\alpha)})_{i,i},
	\end{align*}
	for $\alpha \in  Q_1, \, i \in \lbrace 1, \ldots, d_{t(\alpha)}\rbrace, \, j \in \lbrace 1, \ldots, d_{h(\alpha)}\rbrace$. 
	
	\begin{rmk}\label{MapWeightsArrows}
		Notice that if $\alpha$ and $\alpha^\prime$ are two distinct arrows in $ Q$ having the same endpoints, then they induce the same weights. In particular, we see that we have a surjection from the set of weights to the set $\bar{ Q}_1$ such that the fibre over the arrow $\alpha$ contains $d_{h(\alpha)} d_{t(\alpha)}$ elements. This map induces a definition of multiplicity for the weights via the analogous notion for arrows of $\bar{ Q}$.
	\end{rmk}
	
	We fix a vertex $\overline{v}$ of $ Q$ and choose an index $\overline{k} \in \lbrace 1, \ldots, d_{\overline{v}} \rbrace$.
	Consider the hyperplane
	\begin{align*}
		\mathfrak{a}_\mathbb{C} := V({u}_{\overline{v}, \overline{k}}) \subset \mathfrak{h}.
	\end{align*}
	
	There is an obvious isomorphism $\mathfrak{a}_{\C} \cong \bigoplus_{v \in  Q_0} \mathbb{C}^{d_v - \delta_{v, \overline{v}}}$, which we use implicitly in the following. Let $V_\mathbb{C}$ be the dual linear space of $\mathfrak{a}_\mathbb{C}$. We denote by $\mathfrak{R}_\mathbb{C}$ the image of the set of roots of $G_d$ under the natural projection $\mathfrak{h}^\ast \rightarrow V_\mathbb{C}$, and similarly by $\mathcal{W}_{\mathbb{C}}$ the image of the set of weights.
	
	\begin{rmk}
		Note that the elements of $\mathfrak{R}_\mathbb{C}$ and of $\mathcal{W}_\mathbb{C}$ have real coefficients with respect to the canonical basis of $V_\mathbb{C}$. So we have real linear spaces $\mathfrak{a}$ $(:= V(x_{\overline{v}, \overline{k}}) \subset \bigoplus_{v \in  Q_0} \mathbb{R}^{d_v})$, $V := \mathfrak{a}^\ast$, and subsets $\mathfrak{R}\subset V$, $\mathcal{W}\subset V$, such that
		\begin{align*}
			\mathfrak{a}_\mathbb{C} \cong \mathfrak{a} \otimes_\mathbb{R} \mathbb{C},\, V_\mathbb{C} \cong V \otimes_\mathbb{R} \mathbb{C}, \, \mathfrak{R}_\mathbb{C} \cong \mathfrak{R} \otimes_\mathbb{R} \mathbb{C},\, \mathcal{W}_\mathbb{C} \cong \mathcal{W} \otimes_\mathbb{R} \mathbb{C}.
		\end{align*}
	\end{rmk}
	
	\begin{definition} An \textit{R-charge} is an element $R \in \text{span}_\mathbb{R} \mathcal{W}$. We will denote with $R_\rho$ the component of $R$ corresponding to the weight $\rho$.  
	\end{definition}
	
	For a fixed $R$-charge, we consider the following affine hyperplanes of $\mathfrak{a}_\mathbb{C}$:
	\begin{align*}
		&H_r := V(r - 1), \, r \in \mathfrak{R}_\mathbb{C},\\
		&H_\rho := V(\rho + R_\rho), \, \rho \in \mathcal{W}_\mathbb{C}.
	\end{align*}
	
	\begin{definition} The set $\mathfrak{M}_{sing} \subset \mathfrak{a}_\mathbb{C}$ is given by points $x$ such that there are at least $\vert d \vert - 1$ linearly independent hyperplanes of the form $H_{\ast}$ above meeting at $x$. We denote by $\mathfrak{A}$ the set $\mathfrak{R}\cup \mathcal{W}$ and, given $x \in \mathfrak{M}_{sing}$, we introduce the set $\mathfrak{A}_x \subseteq V$ consisting of elements of $\mathfrak{A}$ such that the corresponding affine hyperplane contains $x$.
	\end{definition}
	
	A stability vector $\theta \in \R  Q_0$ is normalised for ($ Q, d$) if it belongs to the hyperplane
	\begin{align*} 
		\mathcal{H} := \big\lbrace \sum_{v \in  Q_0} d_v \theta_v = 0 \big\rbrace \subset \R  Q_0. 
	\end{align*}
	With this normalisation, $d$ is $\theta$-coprime if, for all $0<e<d$, we have $\theta(e)\neq 0$.
	
	Let $\theta \in \mathbb{R}  Q_0 $ be a stability vector for $ Q$. We extend it to an element $\tilde{\theta} \in V$ by
	\begin{align}\label{DiagonalEmbedding}
		\tilde{\theta}(e_{v,k}) := \theta_v. 
	\end{align}
	
	\begin{definition}
		Let $\zeta$ be a $\Sigma \mathfrak{A}$-regular element belonging to the same chamber of $V$ as $-\tilde{\theta}$. We define the global Jeffrey-Kirwan residue 
		\begin{align*}
			\jk\!: \text{Mer}(\mathfrak{a}_\mathbb{C}) \rightarrow \mathbb{C}
		\end{align*}
		given by 
		\begin{align*}
			\jk(f, \zeta) = \sum_{x \in \mathfrak{M}_{sing}} \text{JK}^{\mathfrak{A}_x}_{\zeta} \left(f\left(\ast + x \right)\right).
		\end{align*}
	\end{definition}
	
	\begin{rmk}
		Clearly, if $\zeta$ is $\mathfrak{A}$-regular, then it is also $\mathfrak{A}_x$-regular for every isolated intersection $x \in \mathfrak{M}_{\text{sing}}$. The converse holds too. Let $\zeta$ be $\mathfrak{A}_x$-regular for every $x \in \mathfrak{M}_{\text{sing}}$. Chosen a set $S \subseteq \mathfrak{A}$ of $\text{dim}(V)-1$ linearly independent elements we can complete it to a basis $B \subseteq \mathfrak{A}$ of $V$. We have that the corresponding affine hyperplanes intersect in a single point $x \in \mathfrak{M}_{\text{sing}}$ for which we clearly have $B \subseteq \mathfrak{A}_x$. Since $\zeta$ is $\mathfrak{A}_x$-regular, $\zeta$ does not belong to the wall $\sum_{s \in S} \mathbb{R} \cdot s$.
	\end{rmk}
	
	\begin{rmk}
		Note that, in general, $\jk(f, \zeta)$ depends on the choice of $\zeta$. However, if $f\in \text{Mer}(\mathfrak{a}_\mathbb{C})$ is such that, for every $x \in \mathfrak{M}_{sing}$, there is a ball $D \subset \mathfrak{a}_\mathbb{C}$ around the origin such that $f$ is holomorphic on the $\mathfrak{A}_x$-regular locus 
		\begin{equation*}
			(x + D)\setminus \bigcup_{l \in \mathfrak{A}_x} V(l(\ast - x)), 
		\end{equation*}
		then, by Theorem \ref{JKFlag}, $\jk(f,\zeta)$ depends only on the chamber containing $\zeta$ and not on the particular choice of sum-regular vector.
	\end{rmk}
	
	\begin{definition} Fix $z \in \mathbb{C}\setminus \mathbb{Z}$. The one-loop factor is the meromorphic function $Z_{1-loop} \in \text{Mer}(\mathfrak{a}_\mathbb{C})$, depending on $z$, given by
		\begin{align*}
			&Z_{1-loop}(u,z) =\\
			& \left(- \frac{\pi z}{\sin(\pi z)} \right)^{\vert d \vert -1} \left( \prod_{r \in \mathfrak{R}_\mathbb{C}} \frac{\sin(\pi z r(u))}{\sin(\pi z( r (u) - 1 ))} \right) \left( \prod_{\rho \in \mathcal{W}_\mathbb{C}} \frac{\sin(\pi z( \rho(u) + R_\rho - 1))}{\sin(\pi z(\rho (u) + R_\rho))} \right)^{m_\rho}.
		\end{align*}
	\end{definition}
	
	Suppose $\tilde{\theta}$ is a $\mathfrak{A}$-regular vector and that $\mathfrak{A}_x$ is projective for every $x \in \mathfrak{M}_{sing}$ (the latter condition holds automatically if $\mathfrak{A}_x$ is a basis; this will always be the case in our applications). The quantity of physical interest (i.e., the $\tau \to {\rm i} \infty$ limit of the corresponding partition function on $T^2 = \C/\bra1, \tau\ket$) is given by
	\begin{equation*}
		\frac{1}{\prod_{v \in  Q_0}  d_v !} \jk\left(Z_{1-loop}\left(u, z \right),\zeta\right),\,\text{ for }   z \in \mathbb{C}\setminus\mathbb{Z}.
	\end{equation*}
	
	If the quiver moduli space $\mathcal{M}^{\theta-sst}_d$ is smooth of dimension $D$, then physics predicts that its Poincar\'e polynomial in singular cohomology $P$ is determined by
	\begin{align*}
		P(e^{i\pi z}) = e^{i \pi z D} \frac{1}{\prod_{v \in  Q_0} d_v !} \jk\left(Z_{1-loop}\left(u, z \right),\zeta\right),\,\text{ for } z \in \mathbb{C}\setminus\mathbb{Z}.
	\end{align*}
	
	More generally, if $\mathcal{M}^{\theta-sst}_d$ is singular, then the Poincar\'e polynomial should be replaced by a suitable refined (Donaldson-Thomas) quiver invariant \cite{js, ks, sven}.
	
	Note that, for $\alpha, \beta \in \mathbb{R}$, we have
	\begin{align*}
		\frac{\sin(\pi z(y+ \alpha))}{\sin(\pi z(y+ \beta))} \xrightarrow{z \rightarrow 0} \frac{y+ \alpha}{y+ \beta},
	\end{align*}
	locally uniformly for $y \in \mathbb{C}^*$. Thus, locally uniformly on $\mathfrak{a}_\mathbb{C}^{reg}$, the one-loop factor $Z_{1-loop}(u, z)$ converges to the meromorphic function      
	\begin{align}\label{functionForEuler}
		\opZ_Q(u) := (-1)^{\vert d \vert -1}\left( \prod_{r \in \mathfrak{R}_\mathbb{C}} \frac{ r(u)}{ r (u) - 1} \right) \prod_{\rho \in \mathcal{W}_\mathbb{C}} \left(\frac{\rho(u) + R_\rho - 1}{\rho (u) + R_\rho} \right)^{m_\rho} \in \text{Mer}(\mathfrak{a}_\mathbb{C}). 
	\end{align}
	\begin{rmk}\label{JKSignRemark}
		Note that this normalisation for $\opZ_Q$ differs from \eqref{MeroForms} by a sign $(-1)^D$, with 
		\begin{align*}
			D	= \sum_{\alpha \in Q_1} d_{t(\alpha)} d_{h(\alpha)} - \vert d \vert + 1. 
		\end{align*}
		In the rest of the paper we follow this sign convention.
	\end{rmk}
	
	\begin{lem} As $z \to 0$, we have
		\begin{equation*}
			\jk\left(Z_{1-loop}\left(u, z \right),\zeta\right) \to \jk(\opZ_Q, \zeta). 
		\end{equation*} 
	\end{lem}
	\begin{proof} Let $x \in \mathfrak{M}_{sing}$. We know that there is an open ball $D \subset \mathfrak{a}_\mathbb{C}$ around $x$ such that $\mathfrak{a}_\mathbb{C}^{reg} \cap D$ is the complement in $D$ of a finite number of hyperplanes (corresponding to the set of functionals $\mathfrak{A}_x$). This implies that there is a product of annuli $A \subset D$, centred at $x$, such that $A \subset \mathfrak{a}_\mathbb{C}^{reg}$, hence satisfying
		\begin{align} 
			Z_{1-loop}(*, z)_{\vert A} \xrightarrow{z\rightarrow 0} Z_Q(*)_{\vert A} 
		\end{align}
		uniformly. Corollary \ref{JKBasisCor} and Proposition \ref{UniformConvIteratedRes} now imply that 
		\begin{align*} 
			\lim_{z \rightarrow 0} \jk_{\zeta}^{\mathfrak{A}_x}(Z_{1-loop}(* + x, z)) = \jk_{\zeta}^{\mathfrak{A}_x}(Z_Q(*+ x)),  
		\end{align*}
		for all $x \in \mathfrak{M}_{sing}$, as required.
	\end{proof}
	In particular, if $\mathcal{M}^{\theta-sst}_d$ is smooth, physics predicts the identity
	\begin{align*}
		\chi\left(\mathcal{M}^{\theta-sst}_d( Q) \right) = \frac{1}{\prod_{v \in  Q_0} d_v !} \jk(\opZ_Q, \zeta).
	\end{align*}
	\begin{rmk}
		This coincides with the conjectural equality (\ref{DTJK}) since in this case $\bar{\chi}(d, \theta) = (-1)^D\chi(\mathcal{M}_d^{\theta-sst})$ (see e.g. \cite{beau}, equation (1.1)).
	\end{rmk}
	More generally, if $\mathcal{M}^{\theta-sst}_d$ is singular, then the topological Euler characteristic should be replaced by a suitable generalised (Donaldson-Thomas) quiver invariant \cite{js, ks}.	
	
	\section{Abelianisation}\label{AbelianSec}
	
	In this Section we discuss abelianisation for JK residues and quiver invariants, following \cite{beau, mps, rsw}. 
	
	Fix a quiver $Q$ without loops or oriented cycles, with dimension vector $d$, and let $i \in Q_0$ denote a fixed vertex. Abelianisation results  are conveniently expressed in terms of a new, infinite ``blown-up" quiver $\widehat{Q}^i$, obtained by replacing the vertex $i$ with vertices $i_{k, l}$, $k, l \geq 1$ (all other vertices are unchanged). Similarly, a single arrow $i \to j$ ($j \to i$) is replaced by $l$ arrows $i_{k, l} \to j$ (respectively $j \to i_{k, l}$) for all $k, l \geq 1$.
	
	A multiplicity vector $m_* \vdash d_i$, that is, a collection of integers $m_l$, $l \geq 1$, satisfying $\sum_l l m_l = d_i$, defines a dimension vector $\widehat{d}$ for $\widehat{Q}^i$ by the rule 
	\begin{equation*}
		\widehat{d}_{i_{k, l}}(m_*) = 
		\begin{cases}
			1,  \quad 1 \leq k \leq m_l\\
			0,  \quad k > m_l.
		\end{cases}  
	\end{equation*}
	Similarly, a stability vector $\zeta$ for $Q$ induces a stability vector $\widehat{\zeta}$ for $\widehat{Q}^i$, defined by
	\begin{equation*}
		\widehat{\zeta}_{i_{k, l}} = l \zeta_i,\, k, l\geq 1.
	\end{equation*}
	Note that if $\zeta$ is normalised then we have $\sum_{k,l} \widehat{d}_{i_{k,l}} \widehat{\zeta}_{i_{k,l}} = \sum_l m_l l \zeta_i = d_i \zeta_i$, hence $\sum_{w \in \widehat{Q}^i_0} \widehat{d}_w \widehat{\zeta}_w = \sum_{v \in Q_0} d_v \zeta_v = 0$ and so $\widehat{\zeta}$ is a normalised stability vector.
	
	The new pair ($\widehat{Q}^i$, $\widehat{d}$) induces new spaces $\widehat{\mathfrak{a}}^i$, $\widehat{V}^i$ and functionals $\widehat{\mathfrak{A}}^i:= \widehat{\mathfrak{R}}^i \cup \widehat{\mathcal{W}}^i$. There is an isomorphism of $\mathbb{R}$-linear spaces induced by the identity $\sum_{v \in Q_0} d_{v} = \sum_{w \in \widehat{Q}_0} \widehat{d}_{w}$,
	\begin{align*}
		\bigoplus_{v \in  Q_0} \mathbb{R}^{d_v} \cong \bigoplus_{w \in \widehat{ Q}_0} \mathbb{R}^{\widehat{d}_w}.
	\end{align*}
	The choice of a reference coordinate $\overline{u}_{s}$ induces an isomorphism between the corresponding subspaces $\mathfrak{a} \cong \widehat{\mathfrak{a}}^i$. 
	We will identify these two linear spaces via this isomorphism. We will work in the second space, where the coordinates associated with a vertex $w \in Q_0\setminus \lbrace i \rbrace$ will be denoted by $u_{w,1}, \ldots, u_{w, d_w}$ and the coordinate associated with the blown-up node $i_{k,l}$ will be denoted by $u_{i,k,l}$. Let us describe the two sets of functionals $\widehat{\mathcal{W}}^i, \widehat{\mathfrak{R}}^i$ on $\mathfrak{a}^i$.
	\begin{prop}
		We have $\widehat{\mathcal{W}}^i = \mathcal{W}$ and $\widehat{\mathfrak{R}}^i \subseteq \mathfrak{R}$, and equality holds if and only if $d_{i} = 1$. In particular $\widehat{\mathfrak{A}}^i \subseteq \mathfrak{A}$.
	\end{prop}
	\begin{proof}
		Let us denote by $\pi$ the projection $\widehat{Q}^i_1 \rightarrow Q_1$. By Remark \ref{MapWeightsArrows} we have a projection $\widehat{\mathcal{W}} \rightarrow \overline{\widehat{ Q}}_1$ whose fibre over the arrow $\alpha$ is the family of functionals
		\begin{align*}
			\widehat{\mathcal{W}}^i_\alpha = \left\lbrace u \mapsto u_{h(\alpha),k} - u_{t(\alpha),j} \, :  k \in \lbrace 1, \ldots, \widehat{d}_{h(\alpha)} \rbrace, j \in \lbrace 1, \ldots, \widehat{d}_{t(\alpha)} \rbrace \right \rbrace.
		\end{align*}
		If the arrow $\alpha$ is between two vertices in $ Q_0 \setminus \lbrace i \rbrace$ then the functionals in the corresponding fibre are precisely those in the fibre $\mathcal{W}_{\pi(\alpha)}$ over $\pi(\alpha)$ for the projection $\mathcal{W}\rightarrow \overline{ Q}_1$. If this is not the case then one (and only one) of $h(\alpha)$, $t(\alpha)$ is a vertex of the form $i_{k,l}$. We consider only the case $t(\alpha) = v \in  Q_0 \setminus \lbrace i\rbrace$, $h(\alpha) = i_{k,l}$ since the other case is completely analogous. The fibre above this arrow is the family of functionals
		\begin{align*}
			\widehat{\mathcal{W}}_\alpha = \left \lbrace u \mapsto u_{i,k,l} - u_{v,j} \, :  j \in \lbrace 1, \ldots, {d}_{t(\alpha)} \rbrace \right \rbrace
		\end{align*}
		which is clearly contained in the fibre $\mathcal{W}_{\pi(\alpha)}$. 
		It is immediate to see that in this case
		\begin{align*}
			\mathcal{W}_{\pi(\alpha)} = \bigcup_{\beta \in \pi^{-1}(\pi (\alpha))} \widehat{\mathcal{W}}_\beta
		\end{align*}
		Let us now consider the elements of $\widehat{\mathfrak{R}}$. They correspond to vertices, meaning that there is a projection $\widehat{\mathfrak{R}}\rightarrow \widehat{ Q}_0$ whose fibre above $w$ is the family of functionals 
		\begin{align*}
			\widehat{\mathfrak{R}}_w := \left\lbrace u \mapsto u_{w,k} - u_{w,j} \, :  k \neq j \in \lbrace 1, \ldots, \widehat{N}_{w} \rbrace\right \rbrace.
		\end{align*}
		If $w \in  Q_0 \setminus \lbrace i \rbrace$, then this is precisely $\mathfrak{R}_w$, the fibre of $\mathfrak{R} \rightarrow  Q_0$ over $w$. On the other hand we have that $\widehat{\mathfrak{R}}_{i_{k,l}} = \emptyset$ since $\widehat{d}_{i_{k,l}} = 1$. 
	\end{proof}
	This proposition shows, in particular, that the hyperplane arrangement defined in the dual space $V$ by $\widehat{\mathfrak{A}}^i$ is contained (strictly if $d_{i}>1$) in the one generated by $\mathfrak{A}$. 
	\begin{cor}
		Let $\xi$ be a ($\Sigma$)$\mathfrak{A}$-regular element of $V$. Then it is ($\Sigma$)$\widehat{\mathfrak{A}}^i$-regular.
	\end{cor}
	
	A physical argument (see \cite{beau} Section 3.4) suggests a remarkable abelianisation identity for JK residues:
	\begin{equation}\label{JKmps}
		\jk(\opZ_Q(d), \zeta) = d_i!\sum_{m_* \vdash d_i} \prod_{l\geq1} \frac{1}{m_l!}\left(\frac{(-1)^{l-1}}{l^2}\right)^{m_l} \jk(\opZ_{\widehat{Q}^i}(\widehat{d}(m_*)), \widehat{\zeta}).
	\end{equation}
	The right hand side is well defined by the previous Corollary.
	
	The analogue of \eqref{JKmps} is known in the case of quiver invariants \cite{mps, rsw}, namely we have
	\begin{equation}\label{DTmps}
		\bar{\chi}_Q(d, \zeta) = \sum_{m_* \vdash d_i} \prod_{l\geq1} \frac{1}{m_l!}\left(\frac{(-1)^{l-1}}{l^2}\right)^{m_l} \bar{\chi}_{\widehat{Q}^i}(d(m_*), \widehat{\zeta}),
	\end{equation}
	so indeed \eqref{JKmps} is compatible with the conjectural identity \eqref{DTJK}. 
	\begin{rmk} However, the physical argument for \eqref{JKmps} is not based on the correspondence with quiver invariants, but rather on an application of the ``Cauchy-Bose identity"
		\begin{equation*}
			\det \frac{1}{\sinh(\mu_i - \nu_j)} = \frac{\prod_{i < j} \sinh(\mu_i - \nu_j)\sinh(\nu_j - \nu_i)}{\prod_{i, j} \sinh(\mu_i - \nu_j)}
		\end{equation*}
		to the product
		\begin{equation*}
			\prod^{d_i}_{s \neq s' = 1} \frac{\sin \pi\hbar(u_{i,s'} - u_{i, s})}{\sin \pi\hbar(u_{i, s} - u_{i,s'}-1)},
		\end{equation*}
		appearing in the 1-loop factor $Z_{1-loop}(u,z)$, which specialises to the factor
		\begin{equation*}
			\prod^{d_i}_{s \neq s' = 1} \frac{u_{i,s'} - u_{i, s}}{u_{i, s} - u_{i,s'}-1}
		\end{equation*}
		appearing in $\opZ_Q(d)$.
	\end{rmk}
	Clearly, applying \eqref{JKmps} to each vertex $i \in Q_0$ makes the gauge group completely abelian: with a straightforward extension of our notation, we have a conjectural identity
	\begin{equation}\label{JKmpsTotal}
		\jk(\opZ_Q(d), \zeta) = \prod_{i\in Q_0}d_i!\sum_{m_* \vdash d} \prod_{l\geq1} \frac{1}{m_{i,l}!}\left(\frac{(-1)^{l-1}}{l^2}\right)^{m_{i,l}} \jk(\opZ_{\widehat{Q}}(\widehat{d}(m_*)), \widehat{\zeta}).
	\end{equation}
	In the present paper, we assume that this conjectural identity holds or, equivalently for our purposes, we take the right hand side of \eqref{JKmpsTotal} as the definition of the relevant residue operation.
	\begin{definition}\label{JKabDef} The abelianised JK residue of the meromorphic form $\opZ_Q(d)$ is given by 
		\begin{equation*}
			\jk_{ab}(\opZ_Q(d), \zeta) = \prod_{i\in Q_0}d_i!\sum_{m_* \vdash d} \prod_{l\geq1} \frac{1}{m_{i,l}!}\left(\frac{(-1)^{l-1}}{l^2}\right)^{m_{i,l}} \jk(\opZ_{\widehat{Q}}(\widehat{d}(m_*)), \widehat{\zeta}).
		\end{equation*}
		Each meromorphic form $\opZ_{\widehat{Q}}(\widehat{d}(m_*))$ is defined on a torus 
		\begin{equation*}
			\mathbb{T}(\widehat{d}(m_*)) = \big(\prod_{i\in \widehat{Q}_0}(\C^*)^{\widehat{d}_i(m_*)}\big)/\C^*.
		\end{equation*}
	\end{definition}
	We spell out the details for the complete bipartite quiver $K(\ell_1, \ell_2)$: this is given by $\ell_1$ sources, with a single arrow connecting each source to each of $\ell_2$ sinks. We endow this with a compatible stability vector, i.e. one with constant value $\zeta_1$ ($\zeta_2$) on all sources (respectively, sinks). Then, $\widehat{Q}$ above may be replaced by a quiver $\cN$ with
	\begin{align*}
		&\cN_0 = \{i_{(w, m)}:(w, m) \in \N^2\} \cup \{j_{(w, m)}:(w, m) \in \N^2\},\\
		&\cN_1 = \{\alpha_1, \ldots, \alpha_{w w'}\!: i_{(w,m)} \to j_{(w',m')}\}.
	\end{align*}
	A dimension vector for $K(\ell_1, \ell_2)$ is the same as a pair of ordered partitions
	\begin{equation*}
		(P_1, P_2) = \big(\sum^{\ell_1}_{i=1} p_{1i}, \sum^{\ell_2}_{j=1} p_{2j}\big).
	\end{equation*}
	Refinements $(k^1, k^2) \vdash (P_1, P_2)$ are defined by
	\begin{equation*}
		p_{1i} = \sum_w w k^1_{w i},\,p_{2j} = \sum_w w k^2_{w j}. 
	\end{equation*}
	The number of entries of weight $w$ in $k^i$ is 
	\begin{equation*}
		m_w(k^i) = \sum^{\ell_i}_{j=1} k^i_{wj},\, i=1, 2.
	\end{equation*}
	Thus, refinements $(k^1, k^2) \vdash (P_1, P_2)$ give dimension vectors for $\cN$:
	\begin{equation*}
		d(k^1, k^2)(q_{(w,m)}) = 
		\begin{cases}
			1   \quad 1 \leq m \leq m_w(k^p)\\
			0   \quad m > m_w(k^p),
		\end{cases}  
	\end{equation*}
	with $p = 1, 2$ according to $q = i, j$. Finally, the stability vector $\widehat{\zeta}$ for $\cN$ is given by
	\begin{equation*}
		\widehat{\zeta}(i_{(w, m)}) = w \zeta_1,\,\widehat{\zeta}(j_{(w, m)}) = w \zeta_2. 
	\end{equation*} 
	Then, the definition of abelianised JK residues \eqref{JKmpsTotal} may be written in the form
	\begin{align}\label{JKmpsBipartite}
		\nonumber&\frac{1}{(P_1, P_2)!}\jk_{ab}(\opZ_{K(\ell_1, \ell_2)}(P_1, P_2), \zeta)\\
		&= \sum_{(k^1, k^2) \vdash (P_1, P_2)} \prod^2_{i=1} \prod^{l_i}_{j=1} \prod_w \frac{(-1)^{k^i_{w,j}(w-1)}}{k^i_{w,j}!w^{2k^i_{w,j}}}\jk(\opZ_{\cN}(d(k^1,k^2)), \widehat{\zeta}).
	\end{align}
	The meromorphic form $\opZ_{\cN}(d(k^1,k^2))$ is defined on the torus 
	\begin{equation*}
		\mathbb{T}(d(k^1,k^2)) = \big(\prod_{q_{(w,m)} \in \cN_0} (\C^*)^{d(k^1,k^2)(q_{(w,m)})}\big)/\C^*. 
	\end{equation*}
	It can be written conveniently by introducing the variables  
	\begin{align*}
		\{u_{(w, m)}, \textrm{ corresponding to } i_{w,m} \in \cN_0;\,
		v_{(w', m')}, \textrm{ corresponding to } j_{w',m'} \in \cN_0\}.
	\end{align*}
	Then, we have
	\begin{align*}
		&\opZ_{\cN}(d(k^1,k^2))\\
		&= (-1)^D\prod_w \prod^{m_w(k^1)}_{m=1}\prod_{w'} \prod^{m_{w'}(k^2)}_{m'=1}\left(\frac{v_{(w', m')}-u_{(w, m)}+1-R^{(w,m)}_{(w',m')}/2}{u_{(w, m)}-v_{(w', m')}+R^{(w,m)}_{(w',m')}/2}\right)^{w w'}.
	\end{align*}
	
	As we already observed, the analogue of \eqref{JKmpsBipartite} for quiver invariants is known:
	\begin{align}\label{mpsBipartite}
		\bar{\chi}_{K(\ell_1, \ell_2)}((P_1, P_2), \zeta) = \sum_{(k^1, k^2) \vdash (P_1, P_2)} \prod^2_{i=1} \prod^{\ell_i}_{j=1} \prod_w \frac{(-1)^{k^i_{w,j}(w-1)}}{k^i_{w,j}!w^{2k^i_{w,j}}}\bar{\chi}_{\cN}(d(k^1,k^2), \widehat{\zeta}).
	\end{align}
	\begin{rmk} The identities \eqref{DTmps}, \eqref{mpsBipartite} were established at the level of motives in \cite{rsw}. Strictly speaking this only implies that \eqref{DTmps}, \eqref{mpsBipartite} hold when the dimension vectors $d$ (respectively, $(|P_1|, |P_2|)$) are primitive. However, given the recent developments in the theory (see \cite{sven}), the same proofs work for the generalised quiver invariant $\bar{\chi}$.
	\end{rmk}
	
		\section{Stable spanning trees}\label{SpanningTreeSec}
	
	\begin{definition}
		A connected quiver ${T}$ is called a \textit{tree} if $|{T}_1| = {T}_0 -1$. If $T$ is also a subquiver of $ Q$ we will say that it is a \textit{tree of $ Q$}. If moreover ${T}_0 =  Q_0$ we say that ${T}$ is a \textit{spanning tree of} $ Q$.
	\end{definition}
	
	It is easy to see that a tree does not contain unoriented loops. Moreover, it contains either a source with a single arrow, or a sink with a single arrow. Note that if $ Q$ admits a spanning tree then it is obviously connected. 
	
	Let 
	\begin{equation*}
		\mathbb{I} = \sum_{v \in  Q_0} v \in \Z_{\geq 0}  Q_0
	\end{equation*}
	denote the full abelian dimension vector.
	\begin{definition}
		Let $ Q$ be a connected quiver and let $T$ be a spanning tree of $ Q$. Note that a stability $\theta$ for $ Q$ induces naturally a stability for $T$. We say that $T$ is \textit{$\theta$-stable} if the corresponding moduli space $\mathcal{M}_{\mathbb{I}}^{\theta-st}(T)$ of $\theta$-stable abelian representations is nonempty. The set of $\theta$-stable spanning trees of $ Q$ will be denoted by $N^\theta( Q)$.
	\end{definition}
	
	Note that the projection map $ Q_1 \rightarrow \bar{ Q}_1$ induces a surjective morphism
	\begin{align*}
		N^\theta( Q) \rightarrow N^\theta(\bar{ Q}).
	\end{align*}
	The fibre over $T \in N^\theta(\bar{ Q})$ has cardinality $\prod_{\alpha \in T} m_\alpha$.
	
	A representation of $ Q$ with dimension vector $\mathbb{I}$ is the same as an element $R \in \C  Q_1$, by associating to an arrow $\alpha$ the linear morphism $\C \xrightarrow{\times R_\alpha} \C$. With a slight abuse of notation, we write $\mathbb{I}$ for the abelian representation that associates the identity morphism with every arrow. The following result is standard.
	\begin{lem}\label{IsoOfRep}
		Let $T$ be a tree and consider an abelian representation $R \in \C T_1$ such that $R_\alpha \neq 0$ for all $\alpha \in T_1$. Then $R \cong \mathbb{I}$. Thus, $\mathcal{M}_{\mathbb{I}}^{\theta-st}(T)$ is either empty or a single point. In particular, a spanning tree $T\subseteq  Q$ is $\theta$-stable if and only if its abelian representation $\mathbb{I}$ is $\theta$-stable.
	\end{lem}

	Recall that the space of normalised stability vectors is given by 
	\begin{align*}
		\mathcal{H} = V\big(\sum_{v \in  Q_0} d_v \theta_v \big) \subset \R  Q_0.
	\end{align*}
	Spanning trees of $ Q$ describe certain natural bases of $\mathcal{H}$. Let $e_v$ be the canonical basis of $\R  Q_0 $. For every arrow $\alpha \in  Q_1$, let $e_\alpha := e_{h(\alpha)} - e_{t(\alpha)} \in \mathcal{H}$.
	
	\begin{prop}\label{BasisForH} Let $S \subset  Q_1$. The set $\lbrace e_{\alpha} :  \alpha \in S \rbrace$ is a basis of $\mathcal{H}$ if and only if $S$ is the set of arrows of a spanning tree $T$ of $ Q$. 
	\end{prop}
	\begin{proof}
		Fix $w \in  Q_0$.\\
		$(\Rightarrow)$ Consider the subquiver $T \subseteq  Q$ obtained by taking as nodes the heads and tails of the elements of $S$ and by considering $S$ as the set of arrows. Assume by contradiction that there is a vertex $v \in  Q_0 \setminus T_0$. This implies that the component of $e_\alpha$ along $e_v$ vanishes for all $\alpha \in S$. Then the vector $e_v - e_w \in \mathcal{H}$ is not contained in $\text{Span}_\mathbb{R}\{e_\alpha : \alpha \in S\}$, which is a contradiction. Hence $T_0 = Q_0$. By assumption, $|S| = \text{dim}_{\mathbb{R}}(\mathcal{H}) = | Q_0 | -1$. Assume by contradiction that $T$ is disconnected and let $C$ be a connected component not containing the vertex $w$. We have the decomposition
		\begin{align}\label{DirectDecomposition}
			\R  Q_0 = \text{Span}_{\mathbb{R}}\{ e_v : v \in C_0 \} \oplus \text{Span}_{\mathbb{R}}\{e_v : v \notin C_0\}.
		\end{align}
		Fix $v \in C_0$ and consider the vector $e_v - e_w \in \mathcal{H}$. We claim that this is not contained in $\text{Span}_\mathbb{R}\{e_\alpha \, : \, \alpha \in S\}$.
		To see this we assume that
		\begin{align*}
			e_v - e_w = \sum_{\alpha \in S} k_\alpha e_\alpha = \sum_{\beta \in C_1} k_\beta e_\beta + \sum_{\gamma \in S \setminus C_1} k_\gamma e_\gamma,
		\end{align*}
		which can be rewritten as
		\begin{align*}
			\big(\sum_{\beta \in C_1} k_\beta e_\beta - e_v \big) + \big(\sum_{\gamma \in S \setminus C_1} k_\gamma e_\gamma - e_w \big) = 0.
		\end{align*}
		By \eqref{DirectDecomposition}, we know that both summands are zero. But this is impossible since, for example, we have $\sum_{\beta \in C_1} k_\beta e_\beta \in \mathcal{H}$ and $e_v \notin \mathcal{H}$.\\
		$(\Leftarrow)$ Since $T$ is a spanning tree for $Q$, we have $| S | = | Q_0 | - 1$, so we just have to prove that $S$ spans $\mathcal{H}$. Fix another vertex $v \in Q_0$ and consider $e_v - e_w \in \mathcal{H}$. The spanning tree $T$ certainly contains a non-oriented path $\alpha_1, \ldots, \alpha_n$ from $w$ to $v$. So there are $\beta_1, \ldots, \beta_n \in \lbrace -1, 1 \rbrace$ such that ${\beta_1} \alpha_1 , \ldots, {\beta_n} \alpha_n$ is an oriented path from $w$ to $v$. Then $e_v - e_w = \sum_{k=1}^n {\beta_k} e_{\alpha_k} \in \text{Span}_\mathbb{R}\{e_\alpha \, : \, \alpha \in S\}$. The claim follows since obviously $\mathcal{H} = \text{Span}_\mathbb{R}\{e_v - e_w \, : \, v \in  Q_0 \}$.
	\end{proof}
	
	The following definition is well-posed by Proposition \ref{BasisForH}.
	\begin{definition}
		We say that $T$ is \textit{$\theta$-regular} if the components of $\theta$ with respect to the basis of $\mathcal{H}$ induced by $T_1$ are all negative. 
	\end{definition}
	
	The main fact we need for our purposes is the following.
	
	\begin{prop}\label{StableAdmissibleProp}
		Let $T$ be a tree and consider a stability vector $\theta \in \mathcal{H}$. Let $\lbrace e_\alpha\rbrace_{\alpha \in T_1}$ be the basis of $\mathcal{H}$ given by Proposition \ref{BasisForH} and consider the corresponding components of $\theta$:
		\begin{align*}
			\theta = \sum_{\alpha \in T_1} c_\alpha e_\alpha.
		\end{align*}
		Then $T$ is $\theta$-stable if and only if $c_\alpha < 0$ for every $\alpha \in T_1$, i.e. if and only if  $\T$ is regular. In particular, we have
		\begin{align*}
			N^\theta( Q) = \left\lbrace \theta\text{-regular spanning trees of }  Q \right\rbrace.
		\end{align*}
	\end{prop}
	\begin{proof} For the sake of the proof, we say that a subquiver $\tilde{ Q} \subseteq  Q$ of a given quiver is \textit{good} if, for all $\alpha \in  Q_1$, we have that $t(\alpha) \in \tilde{ Q}$ implies $h(\alpha) \in \tilde{ Q}_0$ and $\alpha \in \tilde{ Q}_1$. Let $ Q$ be a quiver and consider the map
		\begin{align*}
			\left\lbrace \substack{\text{subrepresentations of the}\\ \text{abelian representation } \mathbb{I}} \right\rbrace \longrightarrow \left\lbrace \text{good subquivers of }  Q \right \rbrace
		\end{align*}
		sending a subrepresentation $R$ to the full subquiver whose nodes are those corresponding to 1-dimensional linear spaces in $R$. One checks that this is a well defined bijection.
		
		Now we claim that if $\tilde{T}$ is a good subquiver of $T$, then 
		\begin{align*}
			\sum_{v \in \tilde{T}_0} \theta_v  = \sum_{\substack{\alpha \in  Q_1\\ h(\alpha) \in \tilde{T}_0,\, t(\alpha) \notin \tilde{T}_0}} c_\alpha.
		\end{align*}
		In order to see this note that
		\begin{align*}
			\theta_v = \sum_{\substack{\alpha \in T_1\\ h(\alpha) = v}}c_\alpha - \sum_{\substack{\beta \in T_1\\ t(\beta) = v}}c_\beta,
		\end{align*}
		therefore
		\begin{align*}
			\sum_{v \in \tilde{T}_0} \theta_v = \sum_{\substack{\alpha \in T_1\\ h(\alpha) \in \tilde{T}_0}}c_\alpha - \sum_{\substack{\beta \in T_1\\ t(\beta) \in \tilde{T}_0}}c_\beta.
		\end{align*}
		Since $\tilde{T}$ is a good subquiver we see that
		\begin{align*}
			t(\beta) \in \tilde{T}_0 \iff h(\beta), t(\beta) \in \tilde{T}_0,
		\end{align*}
		hence the equation above becomes
		\begin{align*}
			\sum_{v \in \tilde{T}_0} \theta_v = \sum_{\substack{\alpha \in T_1\\ h(\alpha) \in \tilde{T}_0}}c_\alpha - \sum_{\substack{\beta \in T_1\\ t(\beta) \in \tilde{T}_0,\, h(\beta) \in \tilde{T}_0}}c_\beta = \sum_{\substack{\alpha \in  Q_1\\ h(\alpha) \in \tilde{T}_0,\, t(\alpha) \notin \tilde{T}_0}} c_\alpha.
		\end{align*}
		($\Leftarrow$) is now clear: if $S$ is a subrepresentation we can consider the associated good subquiver $\tilde{T}$ and notice that
		\begin{align*}
			\theta(S) = \sum_{v \in \tilde{T}_0} \theta_v = \sum_{\substack{\alpha \in  Q_1\\ h(\alpha) \in \tilde{T}_0,\, t(\alpha) \notin \tilde{T}_0}} c_\alpha < 0.
		\end{align*}
		($\Rightarrow$) Let $\alpha \in T_1$ and consider the full subquiver $\tilde{T}$ given by the connected component containing $h(\alpha)$ after removing $\alpha$. If we apply the formula above to this good subquiver we obtain
		\begin{align*}
			c_\alpha = \sum_{v \in \tilde{T}_0} \theta_v < 0.
		\end{align*}
	\end{proof}
		
		\section{JK residues and abelian representations}\label{AbelianJKSec}
	
	In this Section we show how the JK residue formula simplifies considerably for the abelian dimension vector $\mathbb{I}$, at least for generic $R$-charges.
	
	Fix $v_0 \in  Q_0$. Then we have $\mathfrak{a} \cong \mathbb{R}\bra Q_0 \setminus \lbrace v_0 \rbrace\ket$. In this case, by definition, there are no roots: $\mathfrak{R} = \emptyset$. The following result characterises the set of weights.
	
	\begin{prop}\label{WeightsArrows}
		Consider the map
		\begin{align*}
			ar\!: \mathcal{W} \longrightarrow \bar{ Q}_1
		\end{align*}
		described in Remark \ref{MapWeightsArrows}. In the case of the abelian dimension vector $\mathbb{I}$, $ar$ is a bijection. Moreover, $ar$ induces a bijection between the set of bases $\mathfrak{B}(\mathcal{W})$ and the set of spanning trees of $\bar{ Q}$. Finally, if $B \in \mathfrak{B}(\mathcal{W})$ and $T$ is the corresponding spanning tree of $ Q$, write
		\begin{align*}
			\tilde{\theta} = \sum_{\rho \in B} c_\rho \rho; \, \theta = \sum_{\alpha \in T_1} c_\alpha e_\alpha.
		\end{align*}
		Then $c_\rho = c_{\text{ar}(\rho)}$ for every $\rho \in B$.
	\end{prop}
	\begin{proof}
		By Remark \ref{MapWeightsArrows} $ar$ is surjective and the cardinality of the fibre at an arrow $\alpha \in \bar{ Q}_1$ is $d_{t(\alpha)} d_{h(\alpha)} = 1$ since $d = \mathbb{I}$. Thus $ar$ is a bijection. Consider the diagram
		\[
		\begin{tikzcd}
			\bar{ Q}_1 \arrow{r}\arrow[d, swap, "\text{ar}^{-1}"] & \mathcal{H} \arrow[d, "\sim"]\\
			\mathcal{W} \arrow[r, hook] & V
		\end{tikzcd}
		\]
		The vertical arrow on the right is the diagonal embedding $\sim$ defined in (\ref{DiagonalEmbedding}). In the abelian case this is clearly an isomorphism of linear spaces which maps $\theta$ to $\tilde{\theta}$, by definition. The upper horizontal map sends the arrow $\alpha$ to the the vector $e_\alpha := e_{h(\alpha)} - e_{t(\alpha)} \in \mathbb{R}Q_0$, while the lower map is simply the inclusion. This diagram is clearly commutative, and this shows that $ar$ maps a basis of $V$ extracted from $\mathcal{W}$ to a basis of $\mathcal{H}$ induced from $\overline{Q}_1$. The latter set of bases is precisely the set of spanning trees of $\overline{Q}$ by Proposition \ref{BasisForH}. The fact that the components of $\theta$ and $\tilde{\theta}$ with respect to corresponding bases are the same follows from the commutativity of the diagram.
	\end{proof}
	
	We will often identify $\mathcal{W}$ and $\bar{ Q}_1$ via $ar$ implicitly. 
	
	Since, in the abelian case, we have $\mathfrak{R} = \emptyset$, the residue $\jk(\opZ_{Q}(\mathbb{I}), \zeta)$ is well defined for stability vectors $\theta$ whose extensions $\tilde{\theta}$ are $\mathcal{W}$-regular. In general, there is a simpler condition which implies regularity.
	
	\begin{prop}
		Let $ Q$ be a quiver with a stability vector $\theta$ and dimension vector $d$. If $d$ is $\theta$-coprime, then $\tilde{\theta}$ is $\mathfrak{A}$-regular.
	\end{prop}
	\begin{proof} Fix a basis $B$ of $V$ extracted from $\mathfrak{A}$. We need to show that the components of $\tilde{\theta}$ with respect to $B$ are all nonzero. Construct a new quiver $T$ having $d_v$ vertices $v_1, \ldots, v_{d_v} \in T_0$ for every vertex $v \in  Q_0$. Given $v, w \in  Q_0$ and two indices $i \in \lbrace 1, \ldots, d_v \rbrace$ $j \in \lbrace 1, \ldots, d_w \rbrace$, there is a single arrow from $v_i$ to $w_j$ in $T$ if and only if the functional $e_{w,j} - e_{v,i}$ belongs to $\mathcal{B}$. We equip this new quiver with the abelian dimension vector $\mathbb{I}$. We see that the corresponding linear space is  canonically isomorphic to $\mathfrak{a}_\mathbb{C}$ and the corresponding set of functionals is precisely $B$. Since $T$ has no multiple arrows and it induces a basis of $V$, it is a tree by Proposition \ref{WeightsArrows}. Note that the vector $\theta^\prime \in \mathbb{Z}T_0$ defined by $\theta^\prime_{v_i} := \theta_v$ for every $v \in  Q_0$, $i \in \lbrace 1, \ldots, d_v \rbrace$ is a normalised stability for ($T, \mathbb{I}$):
		\begin{align*}
			\sum_{w \in T_0} \theta^\prime_{w} = \sum_{v \in  Q_0} d_v \theta_v = 0.
		\end{align*}
		Moreover we have $\tilde{\theta^\prime} = \tilde{\theta}$. By construction, for every functional $f \in B$ there is an arrow $\alpha_f$ in $T$ and $e_{f} = e_{\alpha_f}$.	
		Let $\tilde{T}$ be the subquiver of $T$ obtained by taking the connected component of $T \setminus \lbrace \alpha_f \rbrace$ containing $h(\alpha_f)$. By the proof of Proposition 8.6, the component of $\tilde{\theta}$ with respect to the element $f \in B$ is given by
		\begin{align*}
			c_f = \sum_{w \in \tilde{T}_0} \theta^\prime_w.
		\end{align*} 
		On the other hand, we can partition the vertices of $\tilde{T}$ into subsets coming from the same vertex of $ Q$. Then the sum above can be rewritten as $\sum_{v \in  Q_0} n_v \theta_v$ with $n_v \leq d_v$ (possibly zero) for every vertex $v$ of $ Q$. Then by definition of coprimality we have that the sum is nonzero.
	\end{proof} 
	The case of trees is especially important for our purposes.
	\begin{prop}\label{SSTProp} Let $T$ be a tree with stability vector $\zeta$, which is regular with respect to $\mathbb{I}$. Then, for representations of $T$ with dimension vector $\mathbb{I}$, $\zeta$-semistability implies $\zeta$-stability.
	\end{prop}
	\begin{proof}
		Consider a $\zeta$-semistable representation $R$, and its subrepresentation $C_\alpha$ obtained by setting to zero the dimension of the linear spaces corresponding to the vertices belonging to the connected component of $T\setminus \lbrace \alpha \rbrace$ containing $t(\alpha)$. Then, by Proposition \ref{StableAdmissibleProp}, we have $\zeta(C_\alpha)  = c_\alpha$, and semistability implies that $c_\alpha \leq 0$. Regularity ensures that equality cannot hold, hence $T$ is $\zeta$-stable, again by Proposition \ref{StableAdmissibleProp}. Then $\mathcal{M}_\mathbb{I}^{\zeta-st}(T) = \lbrace [\mathbb{I}] \rbrace$ is a singleton. We know that $\mathcal{M}_\mathbb{I}^{\zeta-st}(T) \subseteq \mathcal{M}_\mathbb{I}^{\zeta-sst}(T)$ is an open inclusion and that it is dense, since the smaller subset is nonempty and the $\zeta$-semistable variety is irreducible.  
		Then $\mathcal{M}^{\zeta-sst}_{\mathbb{I}}$ is an integral complex variety with an open point, hence it is a singleton and it coincides with the $\zeta$-stable variety.
	\end{proof}
	
	The following result is straightforward.
	
	\begin{lem}\label{NonDegenerateIntersections}
		There is a dense open cone $\mathfrak{P} \subset \R\mathcal{W}$ such that, for $R \in \mathfrak{P}$, all the isolated intersections of the hyperplanes of the form 
		\begin{align*}
			V(\rho + R_\rho),\, \rho \in \mathcal{W}
		\end{align*}
		are non-degenerate, that is, no $|Q_0|$ hyperplanes contain such intersection.
	\end{lem}
	
	Now, for such generic $R$-charges, $R \in \mathfrak{P}$, we will apply the general result Corollary \ref{JKBasisCor} in order to compute the (global) Jeffrey-Kirwan residue.
	
	Let $\phi_x: \mathcal{M}_{\mathfrak{a}_\mathbb{C}, x} \xrightarrow{\sim} \mathcal{M}_{\mathbb{C}^{|Q_0|-1}, 0}$ be the isomorphism induced by the choice of an ordering of the basis $\mathfrak{A}_x$, composed with a translation.  
	\begin{prop}\label{BETHAbelian}
		Suppose $R \in \mathfrak{P}$ and let $f \in \operatorname{Mer}(\mathfrak{a}_\mathbb{C})$ be such that, for all $x \in \mathfrak{M}_{sing}$, there is an open ball $D$ around $x$ with
		\begin{align*}
			f \in \operatorname{Hol}\big(D \setminus \bigcup_{g \in \mathfrak{A}_x} V(g(\ast - x)) \big).
		\end{align*}
		Then, we have
		\begin{align*}
			\operatorname{JK}(f, \zeta) = \sum_{x \in \mathfrak{M}_{sing}} \tilde{\epsilon}(x) \operatorname{IR}_0 ( \phi_x(f(\ast )))
		\end{align*}
		where
		\begin{align*}
			\tilde{\epsilon}(x) := \begin{cases}
				1 & \text{if the components of }\tilde{\theta} \text{ w.r.t. } \mathfrak{A}_x \text{ are all negative},\\
				0 & \text{otherwise}.
			\end{cases}
		\end{align*}
	\end{prop}
	\begin{proof}
		Since $R \in \mathfrak{P}$, all the isolated singularities, i.e. points of $\mathfrak{M}_{sing} \subset \mathfrak{a}_{\C}$, are non-degenerate. This is equivalent to saying that $\mathfrak{A}_x$ is a basis of $V$ for all $x \in \mathfrak{M}_{sing}$. The claim now follows by applying Corollary \ref{JKBasisCor} to all isolated singularities.
	\end{proof}
	
	Using Proposition \ref{WeightsArrows}, this computation can be rephrased in terms of stable spanning trees. Let $T$ be a spanning tree of $\bar{ Q}$. We denote by $x_T \in \mathfrak{M}_{sing}$ the point corresponding to the element of $\mathfrak{B}(\mathcal{W})$ associated with $T$ via Proposition \ref{WeightsArrows}. 
	
	\begin{definition} We let $\phi_{x_T}\!: \mathcal{M}_{\mathfrak{a}_\mathbb{C}, x_T} \rightarrow \mathcal{M}_{\mathbb{C}^n, 0}$ be the isomorphism induced by the choice of an ordering of the basis associated with $T$, composed with a translation.
	\end{definition}
	\begin{cor}\label{BETHAbelianTrees}
		Let $R \in \mathfrak{P}$ and let $f \in \operatorname{Mer}(\mathfrak{a}_\mathbb{C})$ be such that, for every $x \in \mathfrak{M}_{sing}$, there is an open ball $D$ around $x$ such that
		\begin{align*}
			f \in \operatorname{Hol}\big(D \setminus \bigcup_{g \in \mathfrak{A}_x} V(g(\ast - x)) \big).
		\end{align*}
		Then,
		\begin{align*}
			\jk(f, \zeta) = \sum_{T \in N^\theta(\bar{ Q})} \operatorname{IR}_0 (  \phi_{x_T}(f(\ast))).
		\end{align*}
	\end{cor}
	
		\section{Proof of Theorem \ref{QuiverChiTheorem}}\label{QuiverChiProofSec}
	Let $Q$ be a quiver without loops or oriented cycles. Note that, in the abelian case, the $1$-loop factor  is given by 
	\begin{align*}
		&Z_{1-loop}(u,z) :=\left(- \frac{\pi z}{\sin(\pi z)} \right)^{|Q_0|-1} \left( \prod_{\rho \in \mathcal{W}_\mathbb{C}} \frac{\sin(\pi z( \rho(u) + R_\rho - 1))}{\sin(\pi z(\rho (u) + R_\rho))} \right)^{m_\rho}.
	\end{align*} 
	Therefore the rational limit $\opZ_Q(\mathbb{I})$ of the 1-loop factor is given by
	\begin{align}\label{functionForAbelianEuler}
		\opZ_Q(\mathbb{I}) = \opZ_{Q, R}(\mathbb{I}) = (-1)^{|Q_0|-1} \prod_{\rho \in \mathcal{W}_\mathbb{C}} \left(\frac{\rho(u) + R_\rho - 1}{\rho (u) + R_\rho} \right)^{m_\rho}.
	\end{align}
	Fix $\bar{R} \in \mathfrak{P}$, and a regular stability vector $\zeta$. Then, according to Corollary \ref{BETHAbelianTrees}, we have, for all $\lambda \in \R_{>0}$,
	\begin{equation}\label{SumOverStables}
		\jk(\opZ_{Q, \lambda\bar{R}}(\mathbb{I}), \zeta) = \sum_{T \in N^\theta(\bar{Q})} \operatorname{IR}_0 (  \phi_{x_T}(\opZ_{Q, \lambda\bar{R}}(\mathbb{I})(u))).
	\end{equation}
	Let $v_1, \ldots, v_{|Q_0|-1}$ denote the linear coordinates of $\mathbb{C}^{|Q_0|-1}$ given by its canonical basis. Then, we have an equality of the form
	\begin{align*}
		\phi_{T}(\opZ_{Q, \lambda\bar{R}}) =  \prod_{i = 1}^{|Q_0|-1} \left(1 - \frac{1}{v_i}\right)^{m_i} \prod_{j = 1}^h \left(1 - \frac{1}{L_j(v_1, \ldots, v_{|Q_0|-1}) +   \lambda\gamma_j(\bar{R})}\right), 
	\end{align*}
	for certain $L_j \in (\mathbb{C}^{|Q_0|-1})^\ast$ and $\gamma_j \in (\R\mathcal{W})^\ast$, $\gamma_j \neq 0$. So, if we choose $\lambda$ sufficiently large, the divisor of poles of $\phi_{T}(\opZ_{Q, \lambda\bar{R}})$ intersects the open ball of radius $1$ in a subset contained in the union of the coordinate hyperplanes. And, on the product of annuli $\left(A_0\left(\frac{1}{2}, \, 1\right)\right)^{|Q_0|-1}$, we have the uniform convergence 		\begin{align*}
		\phi_{T}(\opZ_{Q, \lambda\bar{R}}) \xrightarrow{\lambda \rightarrow +\infty} \prod_{i = 1}^{|Q_0|-1} \left(1 - \frac{1}{v_i}\right)^{m_i}.
	\end{align*}
	Then, by Proposition \ref{UniformConvIteratedRes}, we have
	\begin{equation}\label{LargeRResidue}
		\lim_{\lambda \to +\infty}\operatorname{IR}_0 (  \phi_{x_T}(\opZ_{Q, \lambda\bar{R}}(\mathbb{I})(u))) = \operatorname{IR}_0\left(\prod_{i = 1}^{|Q_0|-1} \left(1 - \frac{1}{v_i}\right)^{m_i}\right). 
	\end{equation}
	
	\begin{proof}[Proof of Theorem \ref{QuiverChiTheorem}] By Definition \ref{JKabDef} and \eqref{SumOverStables}, we have
		\begin{align*}
			&\frac{1}{d!}\jk_{ab}(\opZ_Q(d), \zeta) = \sum_{m_* \vdash d} \prod_{l\geq1} \frac{1}{m_{i,l}!}\left(\frac{(-1)^{l-1}}{l^2}\right)^{m_{i,l}} \jk(\opZ_{\widehat{Q}}(\widehat{d}(m_*)), \widehat{\zeta})\\
			&=\sum_{m_* \vdash d} \prod_{l\geq1} \frac{1}{m_{i,l}!}\left(\frac{(-1)^{l-1}}{l^2}\right)^{m_{i,l}} \sum_{T \in N^\theta(\bar{Q}')} \operatorname{IR}_0 (  \phi_{x_T}(\opZ_{Q', \lambda\bar{R}}(\mathbb{I})(u))),
		\end{align*}
		where $Q' = Q'(m_*)$ denotes the subquiver of $\widehat{Q}$ with support $\widehat{d}(m_*)$. Passing to the limit, using \eqref{LargeRResidue}, we find  
		\begin{align*}
			\frac{1}{d!}\jk^{\infty}_{ab}(\opZ_Q(d), \zeta) &= \sum_{m_* \vdash d} \prod_{l\geq1} \frac{1}{m_{i,l}!}\left(\frac{(-1)^{l-1}}{l^2}\right)^{m_{i,l}} \sum_{T \in N^\theta(\bar{Q}')}\operatorname{IR}_0\left(\prod_{a \in T_1} \left(1 - \frac{1}{v_{a}}\right)^{m_a}\right)\\
			&= \sum_{m_* \vdash d} \prod_{l\geq1} \frac{1}{m_{i,l}!}\left(\frac{(-1)^{l-1}}{l^2}\right)^{m_{i,l}} |N^\theta(\bar{Q}'(m_*))| \prod_{a \in Q'_1} m_a .    
		\end{align*}
		According to a result of Weist, proved in \cite{weist} for abelian quivers using torus localisation, we have
		\begin{equation*}
			|N^\theta(\bar{Q}'(m_*))| \prod_{a \in Q'_1} m_a = \chi(\mathcal{M}_\mathbb{I}^{\theta-st}(Q'(m_*))).  
		\end{equation*}
		Thus, we find
		\begin{align*}
			\frac{1}{d!}\jk^{\infty}_{ab}(\opZ_Q(d), \zeta) &= \sum_{m_* \vdash d} \prod_{l\geq1} \frac{1}{m_{i,l}!}\left(\frac{(-1)^{l-1}}{l^2}\right)^{m_{i,l}} \chi(\mathcal{M}_{\widehat{d}(m_*)}^{\theta-st}(\widehat{Q})).    
		\end{align*}
		By Proposition \ref{SSTProp}, we have
		\begin{equation*}
			\chi(\mathcal{M}_{\widehat{d}(m_*)}^{\theta-st}(\widehat{Q})) = \chi(\mathcal{M}_{\widehat{d}(m_*)}^{\theta-sst}(\widehat{Q})) = (-1)^D \bar{\chi}_{\widehat{Q}}(\widehat{d}(m_*), \widehat{\zeta}).
		\end{equation*}
		So,
		\begin{align*}
			\frac{1}{d!}\jk^{\infty}_{ab}(\opZ_Q(d), \zeta) &= \sum_{m_* \vdash d} \prod_{l\geq1} \frac{1}{m_{i,l}!}\left(\frac{(-1)^{l-1}}{l^2}\right)^{m_{i,l}} \bar{\chi}_{\widehat{Q}}(\widehat{d}(m_*), \widehat{\zeta})\\
			&= (-1)^D \bar{\chi}_Q(d, \zeta),    
		\end{align*} 
		by abelianisation for quiver invariants, \eqref{DTmps}. Due to our current sign convention for JK residues (see Remark \ref{JKSignRemark}), this is the claim of Theorem \ref{QuiverChiTheorem}.
	\end{proof}
		
	\section{Recollections on the GHK family}\label{GHKSec}
	
	Here we briefly summarise some key aspects of \cite{ghk}, Sections 1-3. Let $(Y, D)$ denote a Looijenga pair such that the intersection matrix $(D_i \cdot D_j)$ is not negative semidefinite. Then, the cone of numerically effective cycles $\NE(Y)_{\R_{\geq 0}}$ is rational polyhedral, so the (toric) monoid $P = \NE(Y)$ is finitely generated. The GHK mirror to $U = Y\setminus D$, as an algebraic holomorphic symplectic surface, exists globally as a family of affine surfaces 
	\begin{equation*}
		\X \to S = \spec \C[P] = \spec\C[\NE(Y)]. 
	\end{equation*}
	
	Initially however $\X$ is constructed as a formal family around the torus fixed point $0 \in S$, corresponding to the maximal monoid ideal $\mathfrak{m} = P \setminus \{0\}$: that is, around the large complex structure limit 
	\begin{equation*}
		\X_0 \cong \mathbb{V}_n := \mathbb{A}^2_{x_1x_2} \cup \cdots \cup \mathbb{A}^2_{x_n x_1} \subset \mathbb{A}^n.
	\end{equation*} 
	It is enough to consider the case when $(Y, D)$ admits a toric model 
	\begin{equation*}
		p\!:(Y, D) \to (\bar{Y}, \bar{D}), 
	\end{equation*}
	blowing up distinct points $x_{ij}$ on $\bar{D}_i$ for $j = 0, \ldots, \ell_i$, with exceptional divisors $E_{ij}$. Here, the base $(\bar{Y}, \bar{D})$ is a toric surface, endowed with a fixed toric structure, with toric anticanonical divisor $\bar{D}$. By convention, we write $\ell_i = -1$ if no points of $\bar{D}_i$ are blown up. 
	
	In the following we write $M$ for the character lattice of $\bar{Y}$ with a fixed identification $M \cong \Z^2$.
	
	Fixing an ample divisor $H$ on $\bar{Y}$, the orthogonal complement $(p^* H)^{\perp}$ with respect to the intersection form is a face of $P$, generated by the classes $[E_{ij}]$, and $G = P \setminus (p^* H)^{\perp}$ is a prime monoid ideal contained in $\mathfrak{m}$. In order to construct the GHK family to all orders around $0 \in S$, it is enough to construct it to all orders around the locus $\spec \C[P]/G \subset S$: indeed the restriction of the family $\X \to S$ to $\spec \C[P]/G$ is the trivial family
	\begin{equation*}
		\X_G \cong \mathbb{V}_n \times \spec \C[P]/G.
	\end{equation*} 
	In fact it is enough to construct $\X \to S$ to all orders around the open torus orbit 
	\begin{equation*}
		T \subset \spec \C[P]/G \subset S,
	\end{equation*}
	known as the Gross-Siebert locus, so we replace $P$ with its localisation along $P\setminus G$ and $\mathfrak{m}$ with the maximal monoid ideal in this localisation. 
	
	Recall that the family $\X \to S$ is constructed from the data of an integral affine surface $B$ with fan $\Sigma$, the tropicalisation of $Y$, and of the canonical scattering diagram $\fD^{\rm can}$ on $B$, defined in terms of relative Gromov-Witten theory on (toric blowups of) $Y$. Nearby the Gross-Siebert locus $T$, however, the family admits a simpler description. The toric model $(Y, D) \to (\bar{Y}, \bar{D})$ induces a canonical piecewise linear map $\nu\!: B \to \bar{B} = \R^2$, identifying $\Sigma$ with the toric fan $\bar{\Sigma}$, and a scattering diagram 
	\begin{equation*}
		\bar{\fD} = \nu(\fD^{\rm can})
	\end{equation*} 
	on the trivial integral affine surface $\bar{B} = \R^2$. For any ideal $I$ which gives an infinitesimal thickening of $T$, i.e., such that $\sqrt{I} = \mathfrak{m}$, we can construct the GHK family $\bar{\X}_{I, \nu(\fD^{\rm can})}$ with data $\bar{B}$, $\nu(\fD^{\rm can})$, and there is an isomorphism of dense open subsets
	\begin{equation*}
		p\!: \X^o_{I, \fD^{\rm can}} \to \bar{\X}^o_{I, \nu(\fD^{\rm can})}
	\end{equation*}
	over $\spec \C[P]/I$ (morally, $p$ is defined away from the singular fibres of the SYZ fibration). Because of consistency on both sides, that is, by the existence of the canonical regular functions $\vartheta_q(t) = \operatorname{Lift}_{t} (q)$ on $\X_{I, \fD^{\rm can}}$ and $\bar{\vartheta}_q(t) = \overline{\operatorname{Lift}}_{t} (q)$ on $\bar{\X}_{I,\nu(\fD^{\rm can})}$, which are defined using broken lines as a key part of the GHK construction, this is enough to determine the family $\X_{I, \fD^{\rm can}} \to S$ from the data of the scattering diagram $\bar{\fD}$. Finally, the latter can be computed as the consistent completion of the initial scattering diagram
	\begin{align*}
		\bar{\fD}_0 &= \{\bar{\rho}_i, \prod^{\ell_i}_{j=0}(1+ b^{-1}_{ij} \bar{X}_{i}),\,i=1,\ldots, n\}\\
		&= \{\R_{\geq 0} m_i, \prod^{\ell_i}_{j=0}(1 + z^{(m_i, \bar{\varphi}(m_i) - [E_{ij}])})\} 
	\end{align*}
	over the Mumford monoid
	\begin{equation*}
		P_{\bar{\varphi}} = \{(m, \bar{\varphi}(m) + p),\,m\in M,\, p \in P\} \subset M \times P;
	\end{equation*}
	that is, we have
	\begin{equation*}
		\bar{\fD} = \nu(\fD^{\rm can})  = \operatorname{Scatter}(\bar{\fD}_0),
	\end{equation*}
	as scattering diagrams for $P_{\bar{\varphi}}$. Here, $m_1, \ldots, m_n$ denote the rays of the toric fan $\bar{\Sigma}$, and $\bar{\varphi}\!: \bar{B} \to P \otimes \R$ is the canonical convex piecewise linear function determined by the bending parameters $\bar{\varphi}(m_i) = p^*[ \bar{D}_i]$.
	
	Another important aspect of the GHK family is given by the periods conjecture, \cite{ghk} Conjecture 0.20. In \cite{ghk}, Section 0.5.4, this is claimed in our case when $(D_i \cdot D_j)$ is not negative semidefinite, and a proof also recently appeared in \cite{LaiZhou}. Thus, in this case, it is known that the local system over the locus of smooth fibres $S^o \subset S$ given by
	\begin{equation*}
		S^o \ni s \mapsto H_2(\X_s, \Z)/\bra \gamma_s \ket,
	\end{equation*}
	where $\gamma_s$ denotes the (monodromy invariant) class of a suitable real torus (the class of the fibre of a Lagrangian fibration), is trivial. Moreover, we can identify each fibre canonically with the lattice
	\begin{equation*}
		\mathcal{Q} = H_2(U, \Z)/\bra \gamma \ket = \bra D_1, \ldots, D_n\ket^{\perp} \subset H_2(Y, \Z)
	\end{equation*}
	for a suitable $2$-torus class $\gamma$. Under this identification, writing $\tilde{\beta}$ for any lift to $H_2(\X_s, \Z)$ of a class $\beta \in \mathcal{Q} \subset H_2(Y, \Z)$, we have
	\begin{equation*}
		z^{\beta} = \exp\big(2\pi {\rm i} \int_{\tilde{\beta}} \Omega\big),
	\end{equation*} 
	as functions on the structure torus $T_Y = \spec \C[H_2(Y, \Z)] \subset S$, provided the holomorphic symplectic form $\Omega$ on the fibres of $\X \to S^o$ is normalised by the condition $\int_{\gamma_s} \Omega = 1$. Thus, the monomials $z^{\beta}$, $\beta \in \mathcal{Q} \subset H_2(Y, \Z)$ are canonical coordinates on the complex moduli space of the surfaces $\X_s$, nearby the large complex structure limit $s \to 0 \in S$. 
	
	Note that mirror to this, by construction, we have  
	\begin{equation*}
		z^{\beta} = \exp\big(2\pi {\rm i} \int_{ \beta } [B + {\rm i}\omega]\big), 
	\end{equation*}
	as $[B + {\rm i}\omega]$ varies in an open subset of the complexified K\"ahler cone on $Y$ (for which the mirror is smooth). 
	
	Let us return to the case of a toric model $(Y, D)\to (\bar{Y}, \bar{D})$, with initial scattering diagram $\bar{\fD}_0 \subset \bar{B} = \R^2$. In order to make direct contact with complex or K\"ahler parameters, we write this in the form
	\begin{equation*}
		\bar{\fD}_0 = \{\R_{\geq 0} m_i, \prod^{\ell_i}_{j=0}(1 + z^{(0, - [E_{i0}])} z^{(m_i, \bar{\varphi}(m_i) + [E_{i0} - E_{ij}])})\},
	\end{equation*}
	where $E_{i0}$ is some fixed choice of a reference exceptional divisor over $\bar{D}_i$. (Note that $z^{(0, - [E_{i0}])}$ is an invertible element in $P_{\bar{\varphi}}$). We have
	\begin{equation*}
		[E_{i0} - E_{ij}] \cdot D_{k} = 0,\, k = 1\ldots, n, 
	\end{equation*}
	so 
	\begin{equation*}
		\beta_{ij} = [E_{i0} - E_{ij}] \in \mathcal{Q},
	\end{equation*}
	and the monomial $z^{\beta_{ij}}$, as a function on $T_Y$, is given by a period or a complexified K\"ahler parameter, 
	\begin{align*}
		z^{\beta_{ij}} &= \exp\big(2\pi {\rm i} \int_{\tilde{\beta}_{ij}} \Omega\big)= \exp\big(2\pi {\rm i} \int_{\beta_{ij}} [B + {\rm i}\omega]\big).
	\end{align*}
	\section{The GHK family and generalised monodromy}\label{GenMonSec}
	Recall the monoid $P$ is given by the localisation of $\NE(Y)$ along the face $(p^*H)^{\perp}$. We consider the Mumford monoid given by 
	\begin{equation*}
		P_{\bar{\varphi}} = \{(m, \bar{\varphi}(m) + p),\,m\in M,\, p \in P\} \subset M \times P,
	\end{equation*}
	with maximal monoid ideal $\mathfrak{m} = P_{\bar{\varphi}} \setminus P^{\times}_{\bar{\varphi}}$. We often identify $M$ with $M \times \{0\} \subset P'_{\bar{\varphi}}$, implicitly. We define a skew-symmetric, integral bilinear form on the lattice $P^{\rm{gp}}_{\bar{\varphi}}$, given by
	\begin{equation*}
		\bra (m, \bar{\varphi}(m) + p), (m', \bar{\varphi}(m') + p')\ket = \bra m , m'\ket
	\end{equation*}
	where $\bra m, m'\ket$ denotes the bilinear form on $M \cong \Z^2$ (with our fixed identification) induced by the determinant. 
	
	The automorphisms $\theta_{\fd}$ attached to the rays $(\fd, f_{\fd})$ of the scattering diagram $\bar{\fD}$ are $\C$-algebra automorphisms of $R_k := \C[P_{\bar{\varphi}}]/\mathfrak{m}^k$, for $k \geq 1$, defined by
	\begin{equation*}
		\theta_{\fd}(z^{(m, \bar{\varphi}(m) + p)}) = z^{(m, \bar{\varphi}(m) + p)} f^{\bra m_{\fd}, m\ket}_{\fd}.
	\end{equation*}
	The initial scattering diagram $\bar{\fD}_0$ is described by the $\C$-algebra automorphisms of $R_k$ acting by
	\begin{equation*}
		\theta_{\bar{\rho}_i, j}(z^{(m,\bar{\varphi}(m) + p)}) = z^{(m,\bar{\varphi}(m) + p)} (1 + z^{(m_i, \bar{\varphi}(m_i) - [E_{ij}])})^{\bra m_i, m\ket},\, j = 0,\ldots,\ell_i. 
	\end{equation*}
	
	Similarly, we can consider the monoid given by
	\begin{equation*}
		P'_{\bar{\varphi}} = \{(m, \bar{\varphi}(m) + v),\,m\in M,\, v \in \bra [E_{i0} - E_{ij}]\ket, \, j = 0,\ldots,\ell_i\} \subset P_{\bar{\varphi}}.
	\end{equation*}
	Then, setting 
	\begin{equation*}
		\tau_i = z^{(0, - [E_{i0}])},\,i=1,\dots,n,
	\end{equation*}
	and, more generally, for later applications,
	\begin{equation*}
		\tau^{(m, \bar{\varphi}(m) + \sum_{i} k_i \beta_{ij})} = \prod_i \tau^{k_i}_i, 
	\end{equation*}
	we can regard $\theta_{\bar{\rho}_i, j}$ as $\C[\![\tau_1,\ldots,\tau_n]\!]$-algebra automorphisms of 
	\begin{equation*}
		R'_k := (\C[P'_{\bar{\varphi}}]/(\mathfrak{m}^k\cap P'_{\bar{\varphi}}))[\![\tau_1,\ldots, \tau_n]\!],
	\end{equation*} 
	acting by
	\begin{equation*}
		\theta_{\bar{\rho}_i, j}(z^{(m,\bar{\varphi}(m) + v)}) = z^{(m,\bar{\varphi}(m) + v)} (1 + \tau_i z^{(m_i, \bar{\varphi}(m_i) + [E_{i0} - E_{ij}])})^{\bra m_i, m\ket}.
	\end{equation*}
	
	We consider the problem of consistent completion for the scattering diagram
	\begin{equation*}
		\bar{\fD}'_0 = \{\R_{\geq 0} m_i, \prod^{\ell_i}_{j=1}(1 + \tau_i z^{(m_i, \bar{\varphi}(m_i) + [E_{i0} - E_{ij}])})\} 
	\end{equation*}
	over $R'_k$; that is, we study
	\begin{equation*}
		\bar{\fD}' = \operatorname{Scatter}(\bar{\fD}'_0) = \{(\mathfrak{d}', f'_{\mathfrak{d}'})\}.
	\end{equation*}
	Clearly, by setting $\tau_i = z^{(0, - [E_{i0}])}$, we recover the scattering diagrams for the toric model $(Y',D') \to (\bar{Y}, \bar{D})$ of the log Calabi-Yau surface $U' = Y'\setminus D'$ which blows up distinct points $x_{ij}$ on $\bar{D}_{i}$ for $i = 1, \ldots, n$, $j = 1,\ldots,\ell_i$. Deformations of $(Y', D')$ can be regarded as deformations of $(Y, D)$ for which the blowup points $x_{i0}$, for $i = 1, \ldots, n$, remain fixed.    
	
	In the Introduction, Section \ref{outline}, we explained and provided ample motivation (as well as several references) for our main technique: describing the scattering process from $\bar{\fD}'_0$ to $\bar{\fD}'$ in terms of the analytic continuation of (flat sections of) a flat connection $\nabla_{\bar{\fD}'_0}$, to (flat sections of) a flat connection $\nabla_{\bar{\fD}'}$. Here we follow closely the reference \cite{fgs} (in particular, see \cite{fgs}, Sections 2 and 3 for introductory material).
	
	Thus, we first consider the problem of constructing a meromorphic connection $\nabla_{\bar{\fD}'_0}$ on the trivial principal $\Aut (R'_k)$-bundle over $\PP^1 = \C \cup \{\infty\}$, for which the generalised monodromy should be given by the automorphisms $\theta_{\bar{\rho}_i, j} \in \Aut(R'_k)$, $j=1,\ldots,\ell_i$, appearing in the initial scattering diagram $\bar{\fD}'_0$, along some corresponding rays $\bar{\rho}_{ij} \subset \C \cong M \otimes \R$. Note that here we use our fixed identification $M \cong \Z^2$ and the standard identification between $\R^2$ and $\C$. The simplest possible type for such a connection requires a simple pole at infinity, and a double pole at $0$. 
	
	In order to apply the results of \cite{fgs}, we choose a homomorphism $\cZ\!:(P'_{\bar{\varphi}})^{\rm{gp}} \to \C$ (a ``central charge''), factoring through a corresponding homomorphism $\cZ\!:\mathcal{Q} \to \C$. 
	\begin{rmk} Such $\cZ$ is in fact a morphism of trivial local systems over the smooth locus $S^o$, from $H_2(\X_s)$ to $\underline{\C}$. However, at this point of the discussion, there is no advantage in specialising $\cZ$ to the periods of $\Omega$ (or to the K\"ahler parameters), that is, it is easier to allow more general, auxiliary central charges.  
	\end{rmk} 
	The homomorphism $\cZ$ can be regarded as a $\C[\![\tau_1,\ldots,\tau_n]\!]$-linear derivation of $R'_k$, by setting $\cZ(z^{\alpha}) = \cZ(\alpha) z^{\alpha}$ for $\alpha \in (P'_{\bar{\varphi}})^{\rm{gp}}$, so, as in \cite{fgs}, we have
	\begin{align}\label{rank1connection}
		\nabla_{\bar{\fD}'_0} = \frac{\cZ}{t^2} + \frac{f}{t},  
	\end{align}
	where $t$ denotes a variable on $\C^* \subset \PP^1$, and $f$ is a suitable derivation of $R'_k$. 
	\begin{prop} There exists a unique such $\nabla_{\bar{\fD}'_0}$ with generalised monodromy given by the automorphisms $\theta_{\bar{\rho}_i, j}$, $j=1,\ldots,\ell_i$ along the ray
		\begin{align*}
			\bar{\rho}_{ij}(\cZ ) &= \R_{> 0} \cZ ((m_i, \bar{\varphi}(m_i) + \beta_{ij})) = \R_{> 0} \cZ_{\bar{s}}(\beta_{ij}).
		\end{align*}
	\end{prop}
	\begin{proof}
		This follows from \cite{fgs}, Theorems 4.1 and 4.2.
	\end{proof} 
	Let us write $\hat{\vartheta}$ for the unique normalised, $\Aut(R'_k)$-valued flat section of $\nabla_{\bar{\fD}'_0}$. By construction, it satisfies the (generalised) monodromy condition
	\begin{equation*}
		\hat{\vartheta}(t^+) = \theta_{\bar{\rho}_{i}, j} \circ \hat{\vartheta}(t^-),     
	\end{equation*}
	along the ray $\bar{\rho}_{ij}(\cZ)$. According to \cite{fgs}, for generic values of $\cZ$, there is an explicit formula for $\hat{\vartheta}$ in terms of a sum over connected, rooted trees $T$, with vertices labelled by elements of the set
	\begin{equation*}
		\LL = \{\Z_{>0}(m_i, \bar{\varphi}(m_i) + \beta_{ij}), i=1,\ldots,n,\,j=1,\ldots,\ell_i\} \subset P'_{\bar{\varphi}}.
	\end{equation*}  
	Given $\alpha \in \LL$, we write $\alpha = k \alpha'$ with $\alpha' \in \LL$ primitive, and set 
	\begin{equation*}
		w(\alpha) = \frac{1}{k^2}.
	\end{equation*} 
	Note that we can regard a rooted tree $T$ as a directed graph, by fixing the unique orientation of the edges which flows away from the root. Then, we define the weight of $T$ as
	\begin{equation*}
		w_{T} = \frac{w(\alpha_T) \alpha_T}{|\Aut(T)|}  \in P'_{\bar{\varphi}}\otimes\Q,
	\end{equation*} 
	where $\alpha_T$ denotes the decoration at the root $i_0 \in T_0$. Similarly, we attach to $T$ a meromorphic function $\operatorname{W}_T(w)$ of the variables $w = \{w_i, \, i \in T_0\}$, given by
	\begin{equation}
		\operatorname{W}_T(w ) = \prod_{\{i \to j\} \in T_1} \frac{w_i}{w_j}  \frac{ \bra w(\alpha_i)\alpha_i, \alpha_j\ket}{w_j - w_i} .
	\end{equation}
	Here, we set $w_{i_0} = w_0$ for the root. We also introduce an integration kernel $\rho_T(t)$, depending on $\cZ$, given by
	\begin{equation}\label{propagators}
		\rho_T(t, w) = \frac{1}{2\pi {\rm i} w_0} \frac{t}{w_0 - t} \prod_{i\in T_0} \frac{1}{2\pi {\rm i}} \exp\left(-\frac{\cZ(\alpha_i)}{w_i}\right), 
	\end{equation} 
	and an integration cycle, also depending on $\cZ$, given by
	\begin{equation*}
		C_T = \{ w_0 \in \R_{> 0} \cZ(\alpha_T);\, w_{j} \in \R_{> 0} \cZ(\alpha_j) \textrm{ for } \{i \to j\} \in T_1\}.
	\end{equation*}
	\begin{prop} Suppose any two rays $\bar{\rho}_{ij}(\cZ)$, $\bar{\rho}_{i'j'}(\cZ)$ with $i \neq i'$ are distinct. Then, for $p \in P'_{\bar{\varphi}}$, we have 
		\begin{equation}\label{explicitY}
			\hat{\vartheta}(t)(z^{p}) = z^p \exp\left(\sum_T \bra p, w_T \ket \prod_{i \in T_{0}} (\tau  z)^{\alpha_i} \int_{C_T} \rho(t, w) \operatorname{W}_T(w) \right), 
		\end{equation}
		where $\int_{C_T} \rho(t, w) \operatorname{W}_T(w)$ denotes the iterated integral along the cycle $C_T$, computed according  to the orientation of $T$.
	\end{prop}
	\begin{proof} This follows from \cite{fgs}, Lemma 4.9.
	\end{proof}
	Let us describe certain holomorphic families of flat connections, specialising to $\nabla_{\bar{\fD}'_0}$. For fixed $r \in \Z/n\Z$, $\bar{\phi} \in (0,\pi/2)$, consider the open subset $S_r(\bar{\phi}) \subset \Hom(\mathcal{Q}, \C)$ such that, for $\cZ \in S_r(\bar{\phi})$, $i \neq r, r+1$, we have 
	\begin{equation*}
		\bar{\rho}_{ij}(\cZ) \subset \{ e^{i \phi} \bar{\rho}_{i},\,\phi \in (-\bar{\phi},\bar{\phi}) \},
	\end{equation*}
	while
	\begin{equation*}
		\bar{\rho}_{rj}(\cZ), \bar{\rho}_{r+1,j}(\cZ) \subset \{ e^{i \phi} \bar{\sigma}_{r, r+1},\,\phi \in (-\bar{\phi},\bar{\phi}) \}, 
	\end{equation*} 
	where $\bar{\sigma}_{r, r+1}$ denotes the corresponding cone of the fan $\bar{\Sigma}$.
	\begin{prop}\label{constMonoProp} Fix $\bar{\phi}$ sufficiently small. Then there is a unique family of connections $\nabla(\cZ)$, of the form \eqref{rank1connection}, parametrised by $S_r(\bar{\phi})$, which has constant generalised monodromy, and which specialises to the connection $\nabla_{\bar{\fD}'_0}$ at a point $\cZ_0$.
	\end{prop}
	\begin{proof} 
		This follows from \cite{fgs} Theorems 4.1 and 4.2.
	\end{proof}
	
	We now specialise our discussion to the case $n = 3$, $\ell_3 = -1$. Thus, the toric base is the surface $\bar{Y} \cong \PP^2$ with standard fan $\{\bar{\rho}_1, \bar{\rho}_2, \bar{\rho}_3\} = \R_{>0}\{m_1, m_2, m_3\}$, and the toric models $(Y, D) \to (\bar{Y}, \bar{D})$ and $(Y', D') \to (\bar{Y}, \bar{D})$ do not blow up points along the toric divisor $\bar{D}_3$. The proper transforms $D_3$, $D'_3$ have positive self-intersection, which implies that the log Calabi-Yau surfaces $U = Y\setminus D$, $U' = Y' \setminus D'$ are affine.
	
	Fix central charges $\cZ_0,\, \cZ^* \in S_r(\bar{\phi})$, satisfying 
	\begin{align*}
		\bar{\rho}_{1j}(\cZ_0) = \bar{\rho}_{2},\,\bar{\rho}_{2j}(\cZ_0) = \bar{\rho}_{1};\,\,\bar{\rho}_{1j}(\cZ^*) = \bar{\rho}_{1},\,\bar{\rho}_{2j}(\cZ) = \bar{\rho}_{2}. 
	\end{align*}
	Recall that the mirror family to $(Y', D')$ is constructed around $\X_0 \cong \mathbb{V}_n$ using the consistent scattering diagram $\bar{\fD}' = \operatorname{Scatter}(\bar{\fD}'_0) = \{(\mathfrak{d}', f'_{\mathfrak{d}'})\}$.
	\begin{cor}\label{isomonoCor} Let $n = 3$, $\ell_3 = -1$. (That is, suppose that the toric model has base $\bar{Y} \cong \PP^2$, and no blowups along $\bar{D}_3$). Consider the family $\nabla(\cZ)$ for $\cZ \in S_1(\bar{\phi})$, with initial value $\nabla(\cZ_0) = \nabla_{\bar{\fD}'_0}$, as in Proposition \ref{constMonoProp}. Then, the (generalised) monodromy of the connection $\nabla(\cZ^*)$ is given by the collection of rays and automorphisms $\{(\mathfrak{d}', \theta_{\mathfrak{d}'} = \theta_{f'_{\mathfrak{d}'}})\}$ of the consistent diagram $\bar{\fD}'$.  
	\end{cor}
	\begin{proof} The consistent completion of $\bar{\fD}'_0$ corresponds to the notion of a constant family of (positive) stability data discussed in \cite{fgs} Sections 2.1-2.3, and thus to the constant generalised monodromy condition, by \cite{fgs} Section 3 and Theorems 4.1, 4.2.   
	\end{proof}
	Let $\hat{\vartheta}(t, \cZ)$ denote the unique normalised, $\Aut(R'_k)$-valued flat section of $\nabla(\cZ)$. Then, $\hat{\vartheta}^* = \hat{\vartheta}(t, \cZ^*)$ satisfies the monodromy condition 
	\begin{equation}\label{monodromyCondition}
		\hat{\vartheta}^*(t^+) = \theta_{\mathfrak{d}'} \circ \hat{\vartheta}^*(t^-),       
	\end{equation}
	along the ray $\mathfrak{d}'$ of $\bar{\fD}'$. Thus, for $p \in P'_{\bar{\varphi}}$, we have the identity of elements of $R'_k$
	\begin{equation}\label{YJump}
		\hat{\vartheta}^*(t^+)(z^p) = \theta_{\mathfrak{d}'}\left(\hat{\vartheta}^*(t^-)(z^p)\right),
	\end{equation}
	along the ray $\mathfrak{d}'$. In particular, this holds for the functions
	\begin{equation*}
		\hat{\vartheta}^*_q = \hat{\vartheta}^*(t)(z^{(q, \bar{\varphi}(q))}), \, q \in M = \bar{B}_0(\Z). 
	\end{equation*} 
	\begin{rmk}\label{GHKGenMonRmk2} 
		The canonical regular functions $\bar{\vartheta}_q$ on the GHK family give another collection of elements of $R'_k$, associated with the scattering diagram $\bar{\fD}'$, satisfying the same identity \eqref{YJump}. Naturally, it would be interesting to compare the functions $\bar{\vartheta}_q$, $\hat{\vartheta}^*_q$. 
	\end{rmk}
	
	For sufficiently small $\bar{\phi}$, let us fix
	\begin{equation*}
		\tilde{t} \in \C \setminus (\{ e^{i \phi} \bar{\rho}_{i},\,\phi \in (-\bar{\phi},\bar{\phi}) \} \cup \{ e^{i \phi} \bar{\sigma}_{r, r+1},\,\phi \in (-\bar{\phi},\bar{\phi}) \}).
	\end{equation*} 
	Then, Corollary \ref{isomonoCor} implies that 
	\begin{equation*}
		\hat{\vartheta}_q(\cZ)= \hat{\vartheta}(\tilde{t}, \cZ)(z^{(q, \bar{\varphi}(q))}), \, q \in M = \bar{B}_0(\Z),
	\end{equation*}
	is a holomorphic function of $\cZ \in S_1(\bar{\phi})$, with values in $R'_k$, thought of as a finite dimensional complex linear space. Therefore, $\hat{\vartheta}^*_q = \hat{\vartheta}_q(\cZ^*) $ can be computed by analytic continuation along a path starting from $\hat{\vartheta}_q(\cZ_0)$. 
	
	Thus, we choose a path $\gamma\!: [0,1] \to S_1(\bar{\phi})$, with $\gamma(0) = \cZ_0,\,\gamma(1) = \cZ^*$, and such that, for $\sigma \in [0,1]$, the rays $\bar{\rho}_{1j}(\cZ_{\gamma(\sigma)})$, $\bar{\rho}_{2j'}(\cZ_{\gamma(\sigma)})$ are distinct, except for a finite, minimal number of critical times $0 < \sigma_1 <  \ldots < \sigma_{a} < 1$, for which we have 
	\begin{equation*}
		\bar{\rho}_{1j_r}(\cZ_{\gamma(\sigma_r)}) = \bar{\rho}_{2j'_r}(\cZ_{\gamma(\sigma_r)}), 
	\end{equation*}
	precisely for $j_r \in \{1, \ldots, \ell_1\}$, $j'_r \in \{1,\ldots, \ell_2\}$.
	
	Note that by \eqref{propagators}, \eqref{explicitY}, the analytic continuation along $\gamma|_{(0, \sigma_1)}$ is uniquely determined as
	\begin{equation*} 
		\hat{\vartheta}_q(\cZ_{\gamma(\sigma)}) = z^{(q, \bar{\varphi}(q))} \exp\left(\sum_T \bra q, w_T \ket \prod_{i \in T_{0}} (\tau z)^{\alpha_i} \int_{C_T(\cZ_{\gamma(\sigma)})} \rho_T(\tilde{t}, w, \cZ_{\gamma(\sigma)}) \operatorname{W}_T(w) \right). 
	\end{equation*} 
	The problem of analytically continuing the expression above across $\sigma_1$, and, inductively, across the subsequent critical times $\sigma_2, \ldots, \sigma_a$, is studied in detail in \cite{fgs}, Section 5. 
	
	As in Section \ref{JKquivSec}, the meromorphic function $\operatorname{W}_T(w)$, specialised at $w_{0} = 0$, is singular along the hyperplane arrangement defined by $(T, i_0, \mathbb{I})$. Fixing any ordering of $T_0$, compatible with the orientation of $T$ as a rooted tree, identifies this with a meromorphic function nearby $0 \in \C^{|T_0|-1}$, which we denote by $\phi(\operatorname{W}_T(w))$. 
	
	Let us write
	\begin{equation*}
		\beta = \sum_{i \in T_0} \alpha_i 
	\end{equation*}
	for the total decoration of $T$. Then, the discussion in \cite{fgs}, Sections 5.1-5.3 yields a representation for the continuation along $\gamma$ of the \emph{iterated integral} of $ \operatorname{W}_T(w)$, 
	\begin{equation*}
		\bra q, w_T\ket \int_{C_T(\cZ)} \rho_T(\tilde{t}, w,\cZ) \operatorname{W}_T(w),
	\end{equation*}
	which contains a distinguished term given by the \emph{iterated residue} of $ \operatorname{W}_T(w)$,
	\begin{equation}\label{topTerm}
		\eps(T) \operatorname{IR}_{0}(\phi(\operatorname{W}_T(w))) \bra q, w_T\ket \frac{1}{2\pi {\rm i}}\int_{\R_{>0} \cZ^*(\beta)}  \frac{\tilde{t}}{w_0 - \tilde{t}} \exp\left(-\frac{\cZ^*(\beta)}{w_0}\right) \frac{d w_0}{w_0},
	\end{equation}
	for a unique $\eps(T) \in \Z$.  
	\begin{rmk} In particular, although we will not use this in the present paper, the procedure explained in \cite{fgs}, Section 5.2 allows to compute the coefficient $\eps(T)$ in an elementary way, essentially by repeated applications of the residue theorem. 
		
		In Proposition \ref{GHKTreeContribution} below we will give an explicit expression for $\eps(T)$, at least in a special case, in terms of semistable representations.
	\end{rmk}
	Moreover, the term \eqref{topTerm} is characterised by its behaviour as a function of $\tilde{t}$: it is precisely the contribution to the analytic continuation which can be extended to a holomorphic function of $\tilde{t} \in \C^* \setminus \R_{>0} \cZ^*(\beta)$, with a branch cut discontinuity along $\R_{>0} \cZ^*(\beta)$. The jump along the latter ray is given by 
	\begin{equation}\label{topTermJump} 
		\frac{ \bra q, w(\alpha_T) \alpha_T\ket}{|\Aut(T)|} \eps(T) \operatorname{IR}_0(\phi(\operatorname{W}_T(w))) \exp\left(-\frac{\cZ^*(\beta)}{\tilde{t}}\right).
	\end{equation}
	It follows in particular that analytic continuation gives an effective procedure to compute the consistent completion of the scattering diagram $\bar{\fD}'_0$. Let us write
	\begin{equation*}
		\log f'_{\mathfrak{d}'} = \sum_{\beta} c_{\beta} z^{\beta}.
	\end{equation*}
	\begin{prop} The coefficient $c_{\beta}$ is determined by the identities 
		\begin{equation}\label{RHscattering}
			\bra q, \beta\ket c_{\beta} = \sum_{T\,:\,\sum_{i}\alpha_i = \beta} \frac{ \bra q, w(\alpha_T) \alpha_T\ket}{|\Aut(T)|} \eps(T) \operatorname{IR}_0(\phi(\operatorname{W}_T(w))),  
		\end{equation}
		for all $q \in M = M \times \{0\} \subset P'_{\bar{\varphi}}$.
	\end{prop}
	\begin{proof} Evaluate \eqref{YJump} on any monomial $z^{(q, \bar{\varphi}(q))}$.
	\end{proof}

	\section{Application of generalised monodromy}\label{GenMonSec2} 
	Recall that, according to \cite{ghk}, Section 3, and by our discussion in Section \ref{GenMonSec}, the mirror family to the log Calabi-Yau surface underlying the Looijenga pair with toric model $(Y', D') \to (\bar{Y}, \bar{D})$ is constructed from the consistent completion $\bar{\fD}' = \{(\mathfrak{d}', f'_{\mathfrak{d}'})\}$ of the initial scattering diagram
	\begin{equation*}
		\bar{\fD}'_0 = \{\R_{\geq 0} m_i, \prod^{\ell_i}_{j=1}(1 + \tau_i z^{(m_i, \bar{\varphi}(m_i) + [E_{i0} - E_{ij}])})\} \subset \bar{B} = \R^2, 
	\end{equation*}
	upon setting $\tau_i = z^{(0, - [E_{i0}])}$. The consistent completion can be computed in terms of the (pushforward of the) canonical scattering diagram $(\mathfrak{D}')^{\rm can} \subset B$: 
	\begin{equation*}
		\operatorname{Scatter}(\bar{\fD}'_0) = \nu_*((\mathfrak{D}')^{\rm can}) 
	\end{equation*}
	(see \cite{ghk}, Proposition 3.26). Namely, we have
	\begin{equation}\label{canonicalDiag}
		\log f_{\mathfrak{d}'} = \nu\left( \sum_{\beta} k_{\beta} N_{\beta} (\tau z)^{\pi_{*}\beta - \varphi_{\tau_{\mathfrak{d}'}}(k_{\beta}m_{\mathfrak{d}'})}\right). 
	\end{equation}
	Here, 
	\begin{equation}\label{relGWDef}
		N_{\beta} = \int_{[\overline{\mathfrak{M}}((\tilde{Y}')^{o}/C^o,\beta)]^{\rm vir}}1
	\end{equation}
	is a relative genus $0$ Gromov-Witten invariant computed on the toric blowup $\pi\!:\tilde{Y}' \to Y'$ corresponding to the ray $\mathfrak{d}'$, with respect to the unique degree $\beta$ such that 
	\begin{equation*}
		\beta \cdot \tilde{D}'_i = 
		\begin{cases}
			k_{\beta} \quad \tilde{D}'_i = C\\
			0 \,\,\,\quad \tilde{D}'_i \neq C, 
		\end{cases}
	\end{equation*}
	for the component $C \subset \tilde{D}'$ which corresponds to the ray $\mathfrak{d}'$ (see \cite{ghk}, Section 1.3 for such toric blowups).
	
	Let us now specialise to the case $n = 3$, $\ell_3 = -1$. Recall in this case the toric base $(\bar{Y}, \bar{\Sigma})$ is given by $\PP^2$ with its canonical fan spanned by $\{m_1, m_2, m_3\}$, and there are no blowups along the toric divisor $\bar{D}_3$. Then, the degree $\beta = \beta_{(P_1, P_2)}$ is determined by a pair of ordered partitions of lengths $\ell_1, \ell_2$,
	\begin{equation*}
		(P_1, P_2) = \big(\sum^{\ell_1}_{i=1} p_{1i}, \sum^{\ell_2}_{j=1} p_{2j}\big),
	\end{equation*}
	satisfying $(|P_1|, |P_2|) = k (a, b)$ for $k > 0$ and primitive $(a, b) \in \Z^2$, through the correspondence
	\begin{equation*}
		H_2(\tilde{Y}', \Z) \ni\beta_{(P_1, P_2)} = \pi^*\beta_k - \sum^{\ell_1}_{i = 1} p_{1i} [E_{1i}] - \sum^{\ell_2}_{j=1} p_{2j} [E_{2j}].
	\end{equation*}
	Here, $\beta_k$ is the pullback to $Y'$ of the unique class on the toric orbifold with fan $\R_{\geq 0}\{m_1, m_2, m_3, m_4 := (a, b)\}$ satisfying
	\begin{equation*}
		\beta_k \cdot \bar{D}_1 = ka,\,\beta_k \cdot \bar{D}_2 = kb,\, \beta_k \cdot \bar{D}_4 = k. 
	\end{equation*}
	The equality $(|P_1|, |P_2|) = k (a, b)$ then implies
	\begin{equation*}
		\beta \cdot \tilde{D}'_1 = 0,\,\beta \cdot \tilde{D}'_2 = 0,\,\beta \cdot \tilde{D}'_3 = 0,\,\beta \cdot \tilde{D}'_4 = k. 
	\end{equation*}
	We will write 
	\begin{equation*}
		N[(P_1, P_2)] = N_{\beta_{(P_1, P_2)}}
	\end{equation*}
	for the corresponding Gromov-Witten invariant \eqref{relGWDef}, with $C = \tilde{D}'_4$. Thus, in this case, the weight function $f'_{\mathfrak{d}'}$ appearing in \eqref{canonicalDiag} can be identified with the formal power series 
	\begin{equation*}
		f'_{(a, b)} \in \C[x,x^{-1}, y, y^{-1}][\![s_1, \ldots, s_{\ell_1}, t_1, \ldots, t_{\ell_2}]\!]
	\end{equation*}
	given by
	\begin{equation}\label{tropVertex}
		\log f'_{(a, b)} = \sum_{k>0} \sum_{(|P_1|, |P_2|) = k (a, b)} k N[(P_1, P_2)] (s, t)^{(P_a, P_b)} (\tau x)^{ka} (\tau y)^{kb}, 
	\end{equation}
	where 
	\begin{equation*}
		(s, t)^{(P_a, P_b)} = \prod^{\ell_1}_{i=1} s^{p_{1i}}_i \prod^{\ell_2}_{j=1} t^{p_{2j}}_j,
	\end{equation*}
	with $s_i$, $t_j$ given by
	\begin{align*}
		&s_i = z^{(0, [E_{10} - E_{1i}])} = \exp\big(2\pi {\rm i} \int_{\tilde{\beta}_{1i}} \Omega\big),\,t_j =  z^{(0, [E_{20} - E_{2j}])} = \exp\big(2\pi {\rm i} \int_{\tilde{\beta}_{2j}} \Omega\big),
	\end{align*}
	and with
	\begin{equation*}
		(\tau x)^{ka} = z^{(k a m_1, - k a [E_{10}])},\,(\tau y)^{kb} = z^{(k b m_2, - k b [E_{20}])}. 
	\end{equation*}
	\begin{rmk} Note that the identity \eqref{tropVertex} is precisely of the type considered in the tropical vertex formalism of \cite{gps}, but here the parameters $s_i$, $t_j$ appearing in \cite{gps} are in fact given by periods of the mirror family of the log Calabi-Yau surface $U' = Y'\setminus D'$.
	\end{rmk}
	Thus, the function $\log f'_{(a, b)}$ can be identified canonically with a corresponding sum over dimension vectors $d$ for $K(\ell_1, \ell_2)$, namely 
	\begin{equation*}
		\log \tilde{f}'_{(a, b)} = \sum_{k > 0} \sum_{|d| = k (a, b)} k c_d z^{d},
	\end{equation*}
	where
	\begin{align}\label{MainProof(i)}
		\nonumber & d = (P_1, P_2),\, c_d = N[(P_1, P_2)],\\
		\nonumber & z^d = (s, t)^{(P_a, P_b)} = z^{(0, \sum^{\ell_1}_{i = 1} p_{1i} [E_{10}-E_{1i}] + \sum^{\ell_2}_{j=1} p_{2j} [E_{20}-E_{2j}])}\\
		& =\exp\big(2\pi {\rm i} \int_{\tilde{\beta}(d)} \Omega\big),
	\end{align}
	for the class 
	\begin{equation*}
		\beta(d) = \sum^{\ell_1}_{i = 1} p_{1i} [E_{10}-E_{1i}] + \sum^{\ell_2}_{j=1} p_{2j} [E_{20}-E_{2j}]) \in\mathcal{Q}.
	\end{equation*}
	This identification establishes Theorem \ref{MainThm} $(i)$ as well as the identity \eqref{MainThmGW} in Theorem \ref{MainThm} $(ii)$.
	
	Next we will describe a procedure which refines the correspondence between the consistent completion of $\bar{\fD}'_0$ and generalised monodromy explained in Section \ref{GenMonSec}. This refinement is needed in order to match abelianisation for quiver invariants.  
	
	The first step is given by the degeneration formula in Gromov-Witten theory, applied to the invariants $N[(P_1, P_2)]$. It is shown in \cite{rsw}, Section 4 that this can be written in the form
	\begin{equation}\label{GWdegeneration}
		N[(P_1, P_2)] = \sum_{(k^1, k^2) \vdash (P_1, P_2)} \prod^2_{i=1} \prod^{\ell_i}_{j=1} \prod_w \frac{(-1)^{k^i_{w,j}(w-1)}}{k^i_{w,j}!w^{k^i_{w,j}}} N^{\rm rel}[(w(k^1), w(k^2))],
	\end{equation}
	where $(w(k^1), w(k^2))$ is a pair of weight vectors (i.e., a pair of sequences of increasing, positive integers), of lengths $\sum_w m_w(k^i)$ for $i =1, 2$, determined by
	\begin{equation*}
		(w(k^i))_j = w, \textrm{ for } j = \sum^{w-1}_{r} m_{r=1}(k^i) +1,\ldots,\sum^{w}_{r=1} m_r(k^i).  
	\end{equation*}
	The weight vectors $(w(k^1), w(k^2))$ encode the orders of tangency of a rational curve in $\bar{Y}$ at specified points of $\bar{D}_1$, $\bar{D}_2$, contained in the smooth locus of $\bar{D}$, and $N^{\rm rel}[(w(k^1), w(k^2))]$ denotes the relative Gromov-Witten invariant virtually enumerating such curves (see \cite{gps}, Section 5 for more details on such relative invariants and on the degeneration formula).
	
	The main result of \cite{rsw}, Section 4 proves an identity between relative Gromov-Witten invariants and quiver invariants, 
	\begin{equation}\label{DTGWrel}
		\bar{\chi}_{\cN}(d(k^1,k^2), \widehat{\zeta}^*) = (-1)^{(P_1, P_2)} N^{\rm rel}[(w(k^1), w(k^2))]\prod^2_{i=1} \prod^{\ell_i}_{j=1} \prod_w w^{k^i_{w,j}},
	\end{equation}
	where the stability vector is given by 
	\begin{equation*}
		\widehat{\zeta}^*(i_{(w, m)}) = w,\,\widehat{\zeta}^*(j_{(w, m)}) = 0,
	\end{equation*}
	and the sign by
	\begin{equation*}
		(-1)^{(P_1, P_2)} = (-1)^{k a b - \sum_i (p_{1 i})^2 - \sum_j (p_{2 j})^2-1}.
	\end{equation*}
	Recall that we also have a ``refined GW/Kronecker" correspondence (see \cite{gp}, \cite{rw}, Section 9 and \cite{bousseauThesis})
	\begin{equation}\label{GWKron}
		N[(P_1, P_2)] = (-1)^{(P_1, P_2)}\bar{\chi}_{K(\ell_1, \ell_2)}((P_1, P_2), \zeta^*),
	\end{equation}
	where 
	\begin{equation*}
		\zeta^*(i) = 1,\, \zeta^*(j) = 0,
	\end{equation*}
	for all sources $i$ (respectively, sinks $j$) in $K(\ell_1, \ell_2)_0$. Thus, \eqref{DTGWrel} shows, in particular, that the abelianisation identity for quiver invariants \eqref{mpsBipartite} corresponds precisely to the Gromov-Witten degeneration formula \eqref{GWdegeneration}.
	\begin{rmk} Again, the identity \eqref{DTGWrel} was only established in \cite{rsw} in the case when $(|P_1|, |P_2|)$ is primitive (where $(k^1,k^2) \vdash (P_1, P_2)$), but the same proof works in general, given the developments of Donaldson-Thomas theory for quivers \cite{sven}. 
	\end{rmk}
	The crucial point for our purposes is that combining \eqref{GWdegeneration} with \eqref{DTGWrel} yields the identity
	\begin{equation}\label{GWBipartiteCorr} 
		N[(P_1, P_2)] = (-1)^{(P_1, P_2)}\sum_{(k^1, k^2) \vdash (P_1, P_2)} \prod^2_{i=1} \prod^{\ell_i}_{j=1} \prod_w \frac{(-1)^{k^i_{w,j}(w-1)}}{k^i_{w,j}!w^{2k^i_{w,j}}} \bar{\chi}_{\cN}(d(k^1,k^2), \widehat{\zeta}^*).
	\end{equation}
	(Equivalently, this can be obtained by combining \eqref{mpsBipartite}, \eqref{GWKron}). In turn, the invariant $(-1)^{(P_1, P_2)}\bar{\chi}_{\cN}(d(k^1,k^2), \widehat{\zeta}^*)$ can be computed in terms of generalised monodromy. 
	
	As in \cite{rsw}, Section 4, we consider the Poisson algebra attached to the quiver $Q \subset \cN$ spanned by the dimension vector $d(k^1, k^2)$. This is the ring 
	\begin{equation*}
		R_Q = \C[\![x_{j_{(w', m')}}, y_{i_{(w, m)}} : i_{(w, m)}, j_{(w', m')} \in Q_0]\!],
	\end{equation*}
	endowed with the Poisson bracket defined by
	\begin{equation*}
		\{x_{j_{(w', m')}}, y_{i_{(w, m)}}\} = \bra j_{(w', m')}, i_{(w, m)} \ket x_{j_{(w', m')}} y_{i_{(w, m)}} = w w' x_{j_{(w', m')}} y_{i_{(w, m)}}
	\end{equation*}
	(while all the other brackets of generators vanish). We introduce the (Poisson) automorphisms 
	\begin{equation*}
		\{\theta_{i_{(w', m')}},\,\theta_{j_{(w', m')}} :  i_{(w, m)}, j_{(w', m')} \in Q_0\} \subset \Aut(R_Q),
	\end{equation*}
	acting by
	\begin{align*}
		&\theta_{j_{(w', m')}}(x_{j_{(w, m)}}) = x_{j_{(w, m)}},\\
		&\theta_{j_{(w', m')}}(y_{i_{(w, m)}}) = y_{i_{(w, m)}}(1+ x_{j_{(w', m')}})^{\bra j_{(w', m')},\,i_{(w, m)} \ket},
	\end{align*}
	respectively
	\begin{align*}
		&\theta_{i_{(w', m')}}(x_{j_{(w, m)}}) = x_{j_{(w, m)}} (1+ y_{i_{(w', m')}})^{\bra i_{(w', m')},\,j_{(w, m)} \ket},\\
		&\theta_{i_{(w', m')}}(y_{i_{(w, m)}}) = y_{i_{(w, m)}}.
	\end{align*}
	Consider the stability vectors $\zeta_0$, $\zeta^*$ for $K(\ell_1, \ell_2)$ given by
	\begin{equation*}
		\zeta_0(i) = 0,\, \zeta_0(j) = 1;\,\zeta^*(i) = 1,\, \zeta^*(j) = 0.  
	\end{equation*}  
	We upgrade these to central charges 
	\begin{equation*}
		\cZ_0,\cZ^* \in \Hom(\Z K(\ell_1, \ell_2)_0, \C),    
	\end{equation*}  
	inducing the same stability conditions. Note that stability vectors $\zeta$ and central charges $\cZ$ have canonical lifts $\widehat{\zeta}$, $\widehat{\cZ}$ to $Q$, given by
	\begin{align*}
		&\widehat{\zeta}(i_{(w, m)}) = w \zeta(i),\,\widehat{\zeta}(j_{(w, m)}) = w \zeta(j),\\
		&\widehat{\cZ}(i_{(w, m)}) = w \cZ(i),\,\widehat{\cZ}(j_{(w, m)}) = w \cZ(j).
	\end{align*}
	In particular, we have
	\begin{align*}
		\widehat{\cZ}(d(k^1, k^2)) &= \sum_{i_{(w, m)} \in Q_0} \widehat{\cZ}(i_{(w, m)}) + \sum_{j_{(w', m')} \in Q_0} \widehat{\cZ}(j_{(w', m')})\\
		&= \sum_{i_{(w, m)} \in Q_0} w \cZ (i) + \sum_{j_{(w', m')} \in Q_0} w \cZ (j) =  \cZ(\beta),
	\end{align*} 
	where $\beta = \sum^2_{i =1} \sum^{\ell_i}_{j=1}p_{ij} \beta_{i j}$.   
	
	Fix a path $\gamma\!:[0, 1] \to \Hom(\Z K(\ell_1, \ell_2)_0, \C)$ with $\cZ_{\gamma(0)} = \cZ_0$, $\cZ_{\gamma(1)} = \cZ^*$. By the results of \cite{fgs}, explained in Section \ref{GenMonSec}, there exists a unique $\Aut(R_Q)$-connection $\nabla_0$, of the form 
	\begin{equation*}
		\nabla_0 = \frac{\widehat{\cZ}_0}{t^2} + \frac{f}{t},
	\end{equation*}
	with generalised monodromy given by the pairs of rays and automorphisms
	\begin{equation*}
		\{\R_{>0} \cZ_0(i), \prod_{i_{(w, m)}\in Q_0} \theta_{i_{(w,m)}}\}, \{\R_{>0} \cZ_0(j), \prod_{j_{(w, m)}\in Q_0} \theta_{j_{(w,m)}}\}.
	\end{equation*}
	Moreover, there is a holomorphic family of connections $\nabla(\cZ)$, parametrised by $\cZ$ in an open neighbourhood of $\gamma([0, 1])$, of the form
	\begin{equation*}
		\nabla(\cZ) = \frac{\widehat{\cZ}}{t^2} + \frac{f(\cZ)}{t},
	\end{equation*}
	with constant generalised monodromy, and with initial value $\nabla(\cZ_{\gamma(0)}) = \nabla_0$. So the canonical $\Aut(R_Q)$-valued flat section $\hat{\vartheta}$ of $\nabla(\cZ_0)$ can be continued analytically along $\gamma$ to a flat section $\hat{\vartheta}^*$ of $\nabla(\cZ^*)$. The section $\hat{\vartheta}$ is given explicitly by \eqref{explicitY}, where now $z^p$ denotes an element of $\C[\Z Q_0]$ (so we have e.g. $z^{j_{(w, m)}} = x_{j_{(w,m)}}$), and $T$ is a rooted tree labelled by elements $\alpha$ of the set
	\begin{equation*}
		\LL = \Z_{>0}\{ i_{(w, m)}, j_{(w', m')} \} \subset \Z_{>0}Q_0.
	\end{equation*}  
	In particular, when continuing $\hat{\vartheta}$ along $\gamma$, we need to consider the analytic continuation, up to a neighbourhood of $\gamma(1) = \cZ^*$, of the graded component appearing in \eqref{explicitY},   
	\begin{equation}\label{ExplicitYabelian}
		\sum_{T\,:\,\sum_{i \in T_0} \alpha_i = d(k^1, k^2)} \bra -, w_T \ket \int_{C_T(\widehat{\cZ}_0)} \rho_T(t, w, \widehat{\cZ}_0) \operatorname{W}_T(w) \prod_{i \in T_{0}} z^{\alpha_i}.
	\end{equation}
	The latter continuation is piecewise holomorphic in $t$, with branch cuts along a finite collection of rays $\rho \subset \C^*$, and the jump along the ray $\R_{>0} \widehat{\cZ}^*(d(k^1, k^2))$ is given by 
	\begin{align}\label{linearForms}
		\sum_{\sum_{i \in T_0} \alpha_i = d(k^1, k^2)}   \bra -, \alpha_T\ket  \eps(T) \operatorname{IR}_0(\phi(\operatorname{W}_T(w))) \exp\left(-\frac{\widehat{\cZ}^*(d(k^1, k^2))}{t}\right) z^{d(k^1,k^2)},
	\end{align}
	for a unique $\eps(T) \in \Z$, determined inductively by the procedure of \cite{fgs}, Section 5.2. We will use an alternative expression for $\eps(T)$, given in Proposition \ref{GHKTreeContribution} below.
	
	On the other hand, by the wall-crossing theory for the generalised quiver invariant $\bar{\chi}$ (see \cite{ks, sven}), and the constant monodromy property, the monodromy of $\nabla(\cZ^*)$ is given by the collection of rays and automorphisms
	\begin{equation*}
		\big\{\rho, \theta_{\rho}:= \exp_{\operatorname{Der}(R_Q)}\big\{-, \sum_{d \in \N Q_0 : \widehat{\cZ}(d) \in \rho} (-1)^d \bar{\chi}_Q(d, \widehat{\zeta}^*) z^d\big\}\big\}, 
	\end{equation*} 
	for suitable signs $(-1)^d$, determined by $(-1)^{d(k^1, k^2)} = (-1)^{(P_1, P_2)}$. As a section of $\nabla(\cZ^*)$, $\hat{\vartheta}^*$ satisfies
	\begin{equation*} 
		\hat{\vartheta}^*(t^+) = \theta_{\rho} \circ \hat{\vartheta}^*(t^-),       
	\end{equation*}
	along the ray $\rho$. Applying this to the ray $\R_{>0} \widehat{\cZ}^*(d(k^1, k^2))$, we see that there is an alternative expression for \eqref{linearForms}, given by 
	\begin{align}\label{linearFormsChi}
		\bra - , (-1)^{(P_1, P_2)} \bar{\chi}_{\cN}(d(k^1,k^2) \widehat{\zeta}^*) d(k^1,k^2) \exp\left(-\frac{\widehat{\cZ}(d(k^1, k^2))}{t}\right)\ket z^{d(k^1, k^2)}.  
	\end{align}
	Comparing \eqref{linearForms} and \eqref{linearFormsChi} yields the identity 
	\begin{align}\label{linearForms2}
		(-1)^{(P_1, P_2)}\bar{\chi}_{\cN}(d(k^1,k^2), \widehat{\zeta}^*) = \sum_{\substack{\alpha_T = i_0 \\ \sum_{i \in T_0} \alpha_i = d(k^1, k^2)}}  \eps(T) \operatorname{IR}_0(\phi(\operatorname{W}_T(w))).
	\end{align}
	
	Note that a tree $T$ appearing in \eqref{linearForms} (or \eqref{linearForms2}) can be identified canonically with a spanning tree of the reduced quiver $\bar{Q}$ (respectively, a spanning tree of $\bar{Q}$ with fixed root). 
	\begin{prop}\label{GHKTreeContribution} We have
		\begin{equation*}
			\eps(T) = \bar{\chi}_T(d(k^1, k^2), \widehat{\zeta}^*).
		\end{equation*}
	\end{prop}
	\begin{proof} This characterisation of $\eps(T)$ can be proved by the trick of applying the results of \cite{fgs} to the tree $T$ itself, thought of as an abstract quiver. 
	\end{proof}
	Comparing \eqref{GWBipartiteCorr} and \eqref{linearForms2} then yields the identities for the scattering diagram coefficients
	\begin{align*}
		& c_d = N[(P_1, P_2)]\\
		& =  \sum_{(k^1, k^2) \vdash (P_1, P_2)} \prod^2_{i=1} \prod^{\ell_i}_{j=1} \prod_w \frac{(-1)^{k^i_{w,j}(w-1)}}{k^i_{w,j}!w^{2k^i_{w,j}}} \sum_{\substack{\alpha_T = i_0 \\ \sum_{i \in T_0} \alpha_i = d(k^1, k^2)}}  \eps(T) \operatorname{IR}_0(\phi(\operatorname{W}_T(w))),
	\end{align*}
	with 
	\begin{equation*}
		\eps(T) = \bar{\chi}_T(d(k^1, k^2), \widehat{\zeta}^*).
	\end{equation*}
	Thus, in view of the expression for abelianised JK invariants \eqref{JKmpsBipartite}, our main claim, the identity \eqref{MainThmJK} in Theorem \ref{MainThm} $(ii)$, follows if we can prove the identity
	\begin{align}\label{JKresW}
		\jk(\opZ_{\cN}(d(k^1,k^2)), \widehat{\zeta}^*) = \sum_{\substack{\alpha_T = i_0 \\ \sum_{i \in T_0} \alpha_i = d(k^1, k^2)}}  \eps(T) \operatorname{IR}_0(\phi(\operatorname{W}_T(w))). 
	\end{align} 
	
	In the next Section we will show how this identity \eqref{JKresW} (and so Theorem \ref{MainThm} $(ii)$), follow from the iterated residue expansion \eqref{JKIntro}. 	 
	
	\section{Completion of the proofs}\label{CompletionSec}
	 
	\begin{proof}[Proof of Theorem \ref{MainThm}] The claim $(i)$ and the identity \eqref{MainThmGW} in $(ii)$ were already established in Section \ref{GenMonSec2}, see in particular \eqref{MainProof(i)}.
		
		It remains to prove \eqref{MainThmJK} in $(ii)$ or, equivalently, as we showed in Section \ref{GenMonSec2}, the identity \eqref{JKresW}.
		
		Consider the bipartite quiver $Q \subset \cN$ spanned by the dimension vector $d(k^1, k^2)$. 
		
		Recall we have shown that for a quiver $Q$, without loops or oriented cycles, not necessarily bipartite, and for general stability vectors $\zeta$, the Jeffrey-Kirwan residue $\jk(\opZ_Q(\mathbb{I}), \zeta)$ with respect to the full abelian dimension vector 
	\begin{equation*}
		\mathbb{I} = \sum_{i \in Q_0} i \in \Z Q_0,
	\end{equation*}
	for generic, fixed $R$-charges $\bar{R} = \bar{R}_{ij}$, $\{i \to j\} \in Q_1$, can be computed as a sum of contributions, one for each spanning tree of $\bar{Q}$, with fixed root. (Recall $\jk(\opZ_Q(\mathbb{I}), \zeta)$ is well defined when $\zeta$ is regular with respect to $\opZ_Q(\mathbb{I})$). It turned out that spanning trees correspond to singular points $x_T$ of the affine hyperplane arrangement defined by $(Q, i_{Q}, \mathbb{I}, \bar{R})$. For each such singular point, there is a specific identification $\phi_T$ of $\opZ_{Q', \bar{R}}(\mathbb{I})(u)$ near $x_T$ with a meromorphic function in a neighbourhood of $0 \in \C^{|Q'_0|-1}$. The contribution of a spanning tree $T \subset \bar{Q}$ is given by  
	\begin{equation*}
		\tilde{\eps}(T) \operatorname{IR}_{0} (  \phi_{x_T}(\opZ_{Q, \bar{R}}(\mathbb{I})(u))), 
	\end{equation*}
	where the weight $\tilde{\eps}(T)$ is determined by the JK residue operation. We proved that 	
	\begin{equation*}
		\tilde{\eps}(T) = \chi(\mathcal{M}^{\zeta-st}_{\mathbb{I}}(T)) = \begin{cases}
			1 & \text{if } T \text{ is stable},\\
			0 & \text{otherwise},
		\end{cases}
	\end{equation*}	
	and moreover, for abelian representations of a tree, the regularity of $\zeta$ implies that $\zeta$-semistability and $\zeta$-stability coincide, so that we have
	\begin{equation*}
		\bar{\chi}_T(\mathbb{I}, \zeta) = \chi(\mathcal{M}^{\zeta-st}_{\mathbb{I}}(T)). 
	\end{equation*}
	Given this, \eqref{JKresW} will follow from the identity
	\begin{equation*}
		\lim_{\lambda \to +\infty} \operatorname{IR}_{0} (  \phi_{x_T}(\opZ_{\cN, \lambda\bar{R}}(d(k^1, k^2))(u)))  = \operatorname{IR}_{0}(\phi(\operatorname{W}_T(w ))) .
	\end{equation*}
Indeed, combining Propositions \ref{GHKTreeContribution} and \ref{SSTProp}, we see that we have 
		\begin{align*}
			\eps(T)\operatorname{IR}_0(\phi(\operatorname{W}_T(w))) = \operatorname{IR}_0\left(\prod_{i = 1}^{|Q_0|-1} m_i \left(1 - \frac{1}{v_i}\right)\right),
		\end{align*}
		if $T$ is stable, while $\eps(T)$ vanishes otherwise. On the other hand, we have
		\begin{equation*}
			\lim_{\lambda \to +\infty} \operatorname{IR}_{0} (  \phi_{x_T}(\opZ_{Q, \lambda\bar{R}}(\mathbb{I})(u)))  = \operatorname{IR}_{0}(\phi(\operatorname{W}_T(w ))),
		\end{equation*}
		by \eqref{LargeRResidue} and the the elementary identity
		\begin{equation*}
			\res_0 \left(1 - \frac{1}{v_i}\right)^{m_i} = \res_0 m_i \left(1 - \frac{1}{v_i}\right). 
		\end{equation*}
		By \eqref{SumOverStables}, this proves the required identity \eqref{JKresW}.  
	\end{proof}
	
	\begin{proof}[Proof of Corollary \ref{MainCor}] By Theorem \ref{MainThm} $(ii)$ we have, with our current sign convention for JK residues (see Remark \ref{JKSignRemark}) 
		\begin{equation*}
			c_d = c_{(P_1, P_2)} = \frac{1}{(P_1, P_2)!}\jk_{ab}(\opZ_{K(\ell_1, \ell_2)}(P_1, P_2), \zeta^*) = N[(P_1, P_2)].  
		\end{equation*}
		Recall the identity \eqref{GWKron}, 
		\begin{equation*}
			N[(P_1, P_2)] = (-1)^{(P_1, P_2)}\bar{\chi}_{K(\ell_1, \ell_2)}((P_1, P_2), \zeta^*).
		\end{equation*}
		Then, we find
		\begin{equation*}
			\frac{(-1)^{(P_1, P_2)}}{(P_1, P_2)!}\jk_{ab}(\opZ_{K(\ell_1, \ell_2)}(P_1, P_2), \zeta^*) = \bar{\chi}_{K(\ell_1, \ell_2)}((P_1, P_2), \zeta^*),
		\end{equation*}
		which coincides with the claim of Corollary \ref{MainCor}.
	\end{proof} 	
	
	\appendix
	
	\section{Iterated residues}\label{iterResSec}
	
	Let $D \subset \mathbb{C}$ be an open disc of radius $r$ centred at the origin. There is an obvious injective morphism of rings
	\begin{align*}
		\sigma_1\!: \text{Hol}(D^n) \rightarrow \text{Mer}(D^{n-1})[\![x_1]\!],\,\sigma_1 f := \sum_{j=0}^{+\infty}  \partial^j_{z_1}  f (0, z_2, \ldots, z_n) x_1^j,
	\end{align*}		
	Thus, by the universal property of localization, there is an induced morphism between the respective fields of fractions:
	\begin{align*}
		\sigma_1\!: \text{Mer}(D^n) \hookrightarrow \text{Mer}(D^{n-1})[\![x_1]\!]_{x_1}.
	\end{align*} 		
	We can define in the same way morphisms $\sigma_2, \ldots, \sigma_{n}$, which yield the extension of fields 
	\begin{align*}
		\sigma\!: \text{Mer}(D^n) \hookrightarrow \mathbb{C}[\![x_n]\!]_{x_n} \cdots [\![x_1]\!]_{x_1}.
	\end{align*}
	This is well defined at the level of germs of meromorphic functions $\M_0$ at $0 \in \C^n$. Fix $\alpha = (\alpha_2, \ldots, \alpha_n) \in D^{n-1}$ and set
	\begin{align}\label{eq1}
		f_\alpha (z) := f(z, \alpha_2, \ldots, \alpha_n).
	\end{align}	
	Let
	\begin{align*}
		l_\alpha := V(z_2-\alpha_2, \ldots, z_n-\alpha_n) \subset D^{n-1}.
	\end{align*}	
	Then, if $f$ is a meromorphic function such that $l_\alpha$ is not contained in the divisor of poles, the coefficients of the formal power series $\sigma_1 f$ evaluated at $\alpha$ equal the coefficients of $\sigma_1 \big( f_\alpha \big)$.		
	\begin{definition} The usual residue map,	$\operatorname{Res}_{x = 0}\!: \C[\![x]\!]_x \longrightarrow \C$, picks the coefficient of $x^{-1}$. Fix $n \in \Z_{>0}$ and consider the composition 
		\begin{align*}
			\mathbb{C}[\![x_n]\!]_{x_n}\cdots[\![x_1]\!]_{x_1} & \xrightarrow{\text{Res}_{x_1 = 0}} \mathbb{C}[\![x_n]\!]_{x_n}\cdots[\![x_{2}]\!]_{x_{2}} \xrightarrow{\text{Res}_{x_2 = 0}} \cdots\\
			\cdots &\xrightarrow{\text{Res}_{x_{n-1}= 0}} \mathbb{C}[\![x_n]\!]_{x_n} \xrightarrow{\text{Res}_{x_n = 0}} \mathbb{C}.
		\end{align*}
		This yields a $\mathbb{C}$-linear morphism, the iterated residue
		\begin{align*}
			\text{IR}_0\!: \M_0 \rightarrow \mathbb{C}.
		\end{align*}
	\end{definition}
	\begin{prop}\label{ResidueAndChangeOfVariables}
		For every biholomorphism of the form
		\begin{align*}
			G\!: V \rightarrow W,\,  G(z_1, \ldots, z_n) := (g_1(z_1), \ldots, g_n(z_n))
		\end{align*}
		where $V, W \subset \mathbb{C}^n$ are open and $0 \in W$, we have  
		\begin{align*}
			\operatorname{IR}_0 (f) = \operatorname{IR}_0 \left( (f \circ (G - G(0)) \cdot \det (J_G)\right)
		\end{align*}
		for all $f \in \mathcal{M}_0$.
	\end{prop}	
	\begin{proof} We may assume that $G(0)=0$ and that $V$ and $W$ are polydiscs, $V = \prod_{i=1}^n V_i$, $W = \prod_{j=1}^n W_j$, with $f$ meromorphic on $W$. Then, there is an open dense subset $U \subset \prod_{j=2}^n W_i$ such that, for all $\alpha \in U$, the line $l_\alpha$ is not contained in the divisor of poles of $f$. Thus, for $\alpha \in U$, $f_\alpha(z)$ is a meromorphic function of one complex variable, so we have
		\begin{align*}
			2 \pi {\rm i} \text{Res}_{x = 0} \left(\sigma_1 (f_\alpha)\right) = \int_{\gamma_\alpha} f_\alpha(z) dz
		\end{align*}
		where $\gamma_\alpha$ is a loop around the origin not containing other singularities of $f_\alpha$. As a consequence,
		\begin{align*}
			&2 \pi {\rm i} \text{Res}_{x_1 = 0} (\sigma_1 (f))(\alpha) 
			= 2 \pi {\rm i} \text{Res}_{x_1 = 0} (\sigma_1 (f_\alpha)) 
			= \int_{\gamma_\alpha} f_\alpha dz\\
			= &\int_{g_1^{-1}(\gamma_\alpha)} (f_\alpha \circ g_1) g_1^\prime dz = \int_{g_1^{-1}(\gamma_\alpha)} (f_1)_\alpha(z) dz
			= 2 \pi {\rm i} \text{Res}_{x_1 = 0} \left(\sigma_1(f_1)\right)(\alpha) 
		\end{align*}
		where $f_1\!: V_1 \times \prod_{j=2}^n W_j \dashrightarrow \mathbb{C}$ is defined by
		\begin{align*}
			f_1(z_1, x_2, \ldots, x_n) := f(g_1(z_1), x_2, \ldots, x_n) g_1^\prime(z_1).
		\end{align*}
		This proves the identity of meromorphic functions on $\prod_{j=2}^n W_j$
		\begin{align*}
			\text{Res}_{x_1 = 0} (\sigma_1 (f)) = \text{Res}_{z_1 = 0} \left(\sigma_1(f_1)\right),
		\end{align*}
		since the equality holds on the dense open subset $U \subset \prod_{j=2}^n W_j$. Inductively, we obtain
		\begin{align*}
			&\operatorname{Res}_{x_n=0}\left(\sigma_n\left(\cdots\text{Res}_{x_2 = 0} \left(\sigma_2\left( \text{Res}_{x_1 = 0} (\sigma_1 (f))\right)\right)\cdots\right)\right)\\ &= \operatorname{Res}_{z_n=0}\left(\sigma_n\left(\cdots\text{Res}_{z_2 = 0} \left(\sigma_2\left( \text{Res}_{z_1 = 0} \left(\sigma_1(f_n)\right)\right)\right)\cdots\right)\right),
		\end{align*} 
		where $f_n\!: V_1 \times V_2 \times \cdots \times V_n \to \mathbb{C}$ is given by
		\begin{align*}
			f_n(z_1, z_2, \ldots, z_n) &:= f(g_1(z_1), g_2(z_2), \ldots, g_n(z_n)) g_1^\prime(z_1) \cdots g_n^\prime(z_n)\\
			&= (f \circ G) \text{det}(J_G).
		\end{align*}		
		The claim follows.
	\end{proof}	
	For our applications, we also need to show that iterated residues respect uniform convergence (near the origin). The case of one variable is given by the following result.
	\begin{lem}
		Let $f_n \in \operatorname{Mer}(D)$ and assume that there is an annulus $A:= A_0(r,R) \subset D$ around the origin such that, for every $n \in \mathbb{N}$, if $f_n$ has a singularity inside the disc $D_0(R)$ then it is a pole at $0$. If we have uniform convergence
		\begin{align*}
			(f_n)_{\vert A} \xrightarrow{n \rightarrow \infty} f_{\vert A}
		\end{align*}
		for some $f \in \operatorname{Mer}(D)$, then
		\begin{align*}
			\lim_{n \rightarrow \infty} \operatorname{Res}_0(f_n) = \operatorname{Res}_0(f)
		\end{align*}
	\end{lem}
	We omit the proof, a straightforward application of the residue theorem.
	\begin{prop}\label{UniformConvIteratedRes}
		Consider a sequence $f_m \in \operatorname{Mer}(D^n)$ whose  divisor of poles is contained in a union of hyperplanes passing through the origin. Assume that there is an annulus $A:= A_0(r,R) \subset D$ such that we have uniform convergence
		\begin{align*}
			(f_m)_{\vert A^n} \xrightarrow{m \rightarrow \infty} f_{\vert A^n}
		\end{align*}
		for some $f \in \operatorname{Mer}(D^n)$. Then
		\begin{align*}
			\lim_{m \rightarrow +\infty} \operatorname{IR}_0 (f_m) = \operatorname{IR}_0 (f).
		\end{align*}
	\end{prop}
	\begin{proof} We work by induction on $n$. The base case $n=1$ follows from the previous Lemma. Assume now the claim holds for $n-1$. Then we can consider the new sequence $g_m \in \operatorname{Mer}(D^{n-1})$, given by
		\begin{align*}
			g_m(y) := \text{Res}_{x_1=0} \circ \sigma_1 (f_{m}(x_1, y)),
		\end{align*}
		and, similarly,  
		\begin{align*}
			g(y) := \text{Res}_{x_1=0} \circ \sigma_1 (f (x_1, y)).
		\end{align*}
		These are meromorphic functions, because $g_m$ is the coefficient of $x_1^{-1}$ in the Laurent expansion of $f_m$ with respect to $x_1$, and similarly for $g$. Moreover, if they have poles, then these lie on some hyperplanes through the origin, for the same reason. By the residue theorem, we find
		\begin{align*}
			\Vert g_m - g\Vert_{A^{n-1}, \infty} \leq \frac{1}{2\pi i} \int_C \Vert f_m - f \Vert_{A^n,\infty} = l(C)\Vert f_m - f \Vert_{A^n,\infty} \rightarrow 0
		\end{align*}
		Therefore $g_m \rightarrow g$ uniformly on $A^{n-1}$. Then we can apply the inductive assumption to $g_m$ and $g$, obtaining the claim.
	\end{proof}

\label{lastpage}

\end{document}